\documentclass{article}

\usepackage{graphicx} 
\usepackage[utf8]{inputenc}
\usepackage{amsmath}
\usepackage{amsthm}
\usepackage{hyperref}
\usepackage{amssymb,graphicx}
\usepackage{tensor}
\usepackage{mathtools}
\usepackage{tikz}
\usepackage{amsmath}
\usepackage{amsfonts}
\usepackage{pdfpages}
\usepackage{mathabx}
\usepackage{mathtools}
\usepackage{amssymb}
\usepackage{hyperref}
\usepackage{textcomp}
\usepackage{amsthm}
\usepackage{mathrsfs}
\usepackage{stmaryrd}
\usepackage{frcursive}
\usepackage[T1]{fontenc}
\usepackage{multirow,tabularx}
\usepackage{ upgreek }
\usepackage{amsmath}
\usepackage{tikz}
\usepackage{amsmath}
\usepackage{caption}
\usepackage{amsmath, amssymb}
\usepackage{amsmath, amsthm, amssymb}
\usepackage{mathtools}
\usepackage{titlesec}
\usepackage{mathtools}
\usepackage{mleftright}

\usepackage{amsthm}
\usepackage{sectsty}
\sectionfont{\centering}
\subsectionfont{\centering}
\newtheorem{remark}{Remark}[section]

\newtheorem{definition}{Definition}[section]
\newtheorem{theorem}{Theorem}[section]
\newtheorem{corollary}{Corollary}[theorem]

\newtheorem{lemma}[theorem]{Lemma}
\newtheorem{prop}{Proposition}[section]

\newtheorem{proposition}[theorem]{Proposition}
\newtheorem{example}[theorem]{Example}

\theoremstyle{definition}

\theoremstyle{remark}

\usepackage{titlesec}
\titleformat{\section}{\normalfont\Large\bfseries}{\thesection}{1em}{}
  
\titlespacing{\section}{0pt}{\parskip}{hhhhh gg y-\parskip}
\titlespacing{\subsection}{0pt}{\parskip}{-\parskip}
\titlespacing{\subsubsection}{0pt}{\parskip}{-\parskip}

\usepackage{setspace}
\theoremstyle{plain}
\usepackage{amsthm}

\title{\textbf{Nilpotency and Engel-Type Theorems for Lie Affgebras}}
\author{{\bf Tarik Anowar$^{1}$,~Ripan Saha$^{2}$\footnote{ Corresponding author:~~Email: ripanjumaths@gmail.com}},~~\bf Sayan Thokdar$^{3}$
        \\
{\small 1.  Department of Mathematics, Raiganj University, Raiganj 733134, West Bengal, India}\\
{\small 2. Department of Mathematics, Raiganj University, Raiganj 733134, West Bengal, India}\\
{\small 3. Department of Mathematics, Indian Institute of Science Education and Research} \\{\small(IISER-Pune), Pune 411008,  Maharashtra, India.}}

\begin{document}
\newpage

\maketitle

\begin{center}
\begin{minipage}{12.cm}
\begin{center}{\bf ABSTRACT}\end{center}
We investigate structural properties of Lie affgebras, viewed as affine analogues of Lie algebras. We introduce the notion of an ideaf, the affine counterpart of an ideal, and establish characterizations of left and right ideafs. We further study centers, quotient structures, and product ideafs, proving, under suitable conditions, that the center of a Lie affgebra is an ideaf. We then develop the notion of affine nilpotency and establish connections between the nilpotency of Lie affgebras and that of their retracted Lie algebras. As an application, we prove an Engel-type theorem for Lie affgebras of the form $\mathfrak{a}(\mathfrak{g};\kappa=2\lambda,\lambda,s)$.

\medskip

\medskip

\end{minipage}
\end{center}
\normalsize\vskip0.5cm
     
\section{INTRODUCTION}

Affine geometry provides a natural framework for studying geometric structures independently of any distinguished origin. Motivated by this principle, the notion of an \emph{affgebra} \cite{And2, Brz222} was introduced as an algebraic counterpart of affine spaces, replacing linear dependence on a base point by intrinsic affine and heap-theoretic operations. This approach offers a basepoint-free formulation of algebraic structures and allows one to extend classical theories from vector spaces to affine spaces in a natural and geometric manner.

The origins of this idea can be traced to the theory of heaps (or torsors)\cite{Pru1, Bae1}, which are algebraic structures equipped with a ternary operation satisfying the Mal’cev and associativity identities. Heaps capture the essence of group structures without requiring a fixed identity element. Every group gives rise to a heap, and conversely, every non-empty heap admits a group structure once a distinguished element is chosen. This close relationship makes heaps a fundamental tool for studying affine and basepoint-free algebraic systems.

In recent years, Brzezi\'{n}ski and his collaborators \cite{Brz2019Truss, BBRS1, Bre1, And1, Brz1, Brz2, Brz3} have developed a systematic theory of affine spaces using heaps and introduced the concept of affgebras \cite{And2, Brz222, Brz26} as algebraic structures defined on affine spaces rather than vector spaces. This framework provides an intrinsic and coordinate-free approach to algebraic geometry and non-associative algebra. Furthermore, the theory of trusses \cite{Brz22}, which enriches heaps with compatible multiplication, has unified several algebraic systems, including groups, rings, and braces.

The study of Lie affgebras \cite{Gra1} originated in the investigation of vector-space-valued Lie brackets on affine spaces and their applications to the differential geometry of affine fibre bundles and frame-independent formulations of Lagrangian mechanics \cite{Tul1, Gra3, Gra2}. Subsequently, a more general approach was proposed in which Lie affgebras are defined without assuming the existence of an underlying vector space. In this setting, vector spaces arise only as tangent or fibre structures associated with chosen points of the affine space. The bi-affine Lie bracket induces Lie algebra structures on these fibres, allowing a Lie affgebra to be viewed as a Lie-algebra-fibred affine space. This perspective reveals deep connections between affine geometry and Lie theory and has led to a systematic study of the relationship between Lie affgebras and their associated Lie algebras.

One of the central results in the classical theory of Lie algebras is Engel’s theorem \cite{Hum}, which characterizes nilpotent Lie algebras in terms of the nilpotency of adjoint representations. This theorem plays a fundamental role in the structure theory of Lie algebras and has far-reaching consequences in representation theory and geometry. Given the close relationship between Lie affgebras and Lie algebras on their fibres, it is natural to ask whether analogous structural results hold in the affine setting.

In this paper, we introduce the notion of an \emph{ideaf}, which may be regarded as the affine analogue of an ideal in Lie algebra theory \cite{Jac1}. We establish several equivalent characterizations of left and right ideafs and investigate their structural properties. Using this notion, we define the center of a Lie affgebra and study abelian Lie affgebras. In particular, we prove that the center $\mathfrak{Z}(\mathfrak{a})$ is an ideaf of a Lie affgebra $\mathfrak{a}$ whenever the characteristic of the base field is not equal to $2$ and $\kappa=2\lambda$, $[x, \lambda(y)]=0,~ \text{for all}~ x, y \in \mathfrak{a}$ and $\text{ad}_s=0$.

We further develop the theory of quotient affgebras via ideafs and introduce the notion of product ideafs. A characterization of product ideafs in terms of ideals of associated Lie algebras is obtained. These results provide a systematic framework for studying substructures and factor structures in the category of Lie affgebras.

Moreover, we define and investigate the concept of affine nilpotency for Lie affgebras and present several illustrative examples. We show that the retract of an affine nilpotent Lie affgebra is a nilpotent Lie algebra. We also establish the converse under certain additional restrictions.
 For a Lie affgebra $\mathfrak{a}(\mathfrak{g}; \kappa, \lambda, s)$, we show that if $\mathfrak{a}(\mathfrak{g}; \kappa, \lambda, s)$ is affinely nilpotent, then all of its elements are affinely ad-nilpotent. Conversely, we prove that if all elements of $\mathfrak{a}(\mathfrak{g}; \kappa=2\lambda, \lambda, s)$ are affinely ad-nilpotent, then $\mathfrak{a}(\mathfrak{g}; \kappa=2\lambda, \lambda, s)$ is an affinely nilpotent Lie affgebra. This result may be viewed as an affine analogue of Engel’s theorem.

At present, we are unable to establish a full Engel-type theorem for arbitrary linear maps $\kappa$ and $\lambda$ satisfying generalized derivation identity in the sense of Leger and Luks \cite{Leg1}. This problem is formulated as an open question and is expected to stimulate further research in the structure theory of Lie affgebras.

The paper is organized as follows. In Section~2, we recall basic notions concerning heaps, affine spaces, and Lie affgebras. Section~3 is devoted to the study of ideafs. In Section~4, we discuss the center and quotient affgebras. In Section~5, we develop the notion of product ideafs and present some related results. In Section~6, we investigate affine nilpotency and affine adjoint maps. In Section~7, we prove our Engel-type theorem for Lie affgebras $\mathfrak{a}(\mathfrak{g}; \kappa=2\lambda, \lambda, s)$. Finally, we discuss open problems and future directions.

\section{PRELIMINARIES}
In this section, we recall some basic definitions and results that will be used throughout the paper. 
Throughout the paper, $\mathbb{K}$ denote an arbitrary field of characteristic not equal to $2$, unless otherwise specified.
\begin{definition} \cite{Pru1,Brz2019Truss, Brz22}
A \emph{heap} is a pair $(H,[\ -\ -\ -])$ consisting of a set $H$ and a ternary operation $[\ -\ -\ -] : H \times H \times H \to H$, $(x, y, z) \mapsto [x, y, z]$, satisfying the following conditions, for all $v, w, x, y, z \in H$:
\begin{enumerate}
    \item $[v, w, [x, y, z]] = [[v, w, x], y, z] $ (Associativity),
    \item $[x, x, y] = y = [y, x, x]$ (Mal’cev identities).
\end{enumerate}

\end{definition}
\begin{definition}
    A heap $(H, [\ -\ -\ -])$ is said to be \emph{Abelian,} if, for all $x, y, z$, ~$[x, y, z] = [z, y, x].$

\end{definition}
\begin{definition}
    A \emph{heap morphism} from $(H, [\ -\ -\ -])$ to $(\widetilde{H}, [\ -\ -\ -])$ is a function $\varphi : H \to \widetilde{H}$
satisfying $\varphi([x, y, z]) = [\varphi(x), \varphi(y), \varphi(z)]$, for all $x, y, z\in H$.

\end{definition}

\begin{definition}\label{Def2.4}\cite{And2,BBRS1, Bre1, Brz1} An \textbf{affine space} over a field $\mathbb{K}$ is a non-empty abelian heap $H$ together with a ternary action $\triangleleft :\mathbb{K}\times H\times H \rightarrow H,$ $(\alpha, x, y)\mapsto \alpha \triangleleft_{x}y$, satisfying the following conditions, for all $x, y, z \in H$ and $\alpha, \beta \in \mathbb{K},$ 
    \begin{enumerate}
        \item $\alpha  \triangleleft_{x}- : H\rightarrow H$ and $-\triangleleft_{x}y:\mathbb{K}\rightarrow H$ are heap homomorphisms
        , where $\mathbb{K}$ is a field, hence is a group, then it can be understood as a heap by the operation $\alpha -\beta + \gamma$;
        \item $(\alpha \beta)\triangleleft_{x}y=\alpha \triangleleft_{x}(\beta\triangleleft_{x}y),$ ($\mathbb{K}$-associativity);
        \item $\alpha \triangleleft_{x}y=<\alpha\triangleleft_{z}y, \alpha\triangleleft_{z}x,x>$, (Base change $x\rightarrow z$ property);
        \item $0\triangleleft_{x}y=x$ $\&$ $ 1\triangleleft_{x}y=y.$
    \end{enumerate}
\end{definition}
\begin{definition}
    An affine map $f: H\rightarrow \widetilde{H}$ is a heap homomorphism preserving the actions in the sense that, for all $\alpha \in \mathbb{K}$, and for all $x, y \in H$,
    $f(\alpha \triangleright_{x}y)=\alpha\triangleright_{f(x)}f(y).$
\end{definition}
\begin{remark}
   The set of all affine endomorphism from $H$ to $\widetilde{H}$ is denoted by $Aff(H, \widetilde{H})$. The standard (traditional) definition of affine space is equivalent to this definition of affine space. If we assume $\mathbb{K}$ to be a commutative ring with unity and $H$ satisfy all the axioms of the definition (\ref{Def2.4}), then we say $H$ is an \emph{affine $\mathbb{K}$-module}. It can be shown that the ternary operation $\alpha \triangleleft_{-}y: H\rightarrow H$ is a heap homomorphism (Middle entry heap homomorphism). 
\end{remark}
\begin{definition}\cite{Brz1}
    Let $\mathbb{K}$ be a field. An \textbf{(associative) $\mathbb{K}$-affgebra} or simply an \textbf{affgebra} is an affine space $H$ together with a bi-affine (associative) multiplication $H\times H\rightarrow H.$
\end{definition}
\begin{lemma}\cite{Brz1}
    Let $f, g, h: H\rightarrow \widetilde{H}$ be homomorphism of affine $\mathbb{K}$-modules. Then $<f, g, h>: H\rightarrow \widetilde{H},$ $x\mapsto <f(x), g(x), h(x)>$ is a homomorphism of affine $\mathbb{K}$-modules.
\end{lemma}
\begin{lemma}\label{0}\cite{Brz1}
     Let $\mathfrak{a}$ be an affine space over $\mathbb{K}$. The set \textbf{Aff($\mathfrak{a}$)} of all affine endomorphisms of $\mathfrak{a}$ is an affgebra with the pointwise heap operation, the action $(\alpha \triangleleft_{f}g)(a)=\alpha \triangleleft_{f(a)}g(a),$ for all $\alpha \in \mathbb{K}$, $f, g \in \textbf{Aff($\mathfrak{a}$)}$, $a\in \mathfrak{a}$, and the multiplication is given by composition.
\end{lemma}

\begin{definition} \cite{And2}
     Let $\mathfrak{a}$ be an \textbf{affine space} over $\mathbb{K}$. A \emph{Lie bracket} on $\mathfrak{a}$ is a binary operation $\{-,-\}:\mathfrak{a} \times \mathfrak{a} \rightarrow \mathfrak{a}$ that satisfies the following conditions:
     \begin{enumerate}
         \item For all $a\in \mathfrak{a},$ both $\{a, -\}$ and $\{-, a\}$ are affine transformations;
         \item Affine antisymmetry: for all $a, b \in \mathfrak{a}$, $<\{a, b\}, \{a, a\}, \{b, a\}>=\{b, b\}$;
         \item Affine Jacobi identity: for all $a, b, c \in \mathfrak{a}$, $<\{a, \{b, c\}\}, \{a, \{a, a\}\}, \{b, \{c,a\}\}, \{b, \{b, b\}\}, \{c, \{a, b\}\}>=\{c, \{c, c\}\}.$
     \end{enumerate}
     An affine space together with an affine Lie bracket is called a Lie affgebra.
 \end{definition}
  \begin{remark}
     The affine antisymmetry condition (2) of the above definition can be equivalently stated as 
     $\{a, b\}=<\{a, a\}, \{b, a\}, \{b, b\}>$.
 \end{remark}
 \begin{proposition}\cite{Brz1}
      An associative $\mathbb{K}$-affgebra $H$ is a Lie affgebra with the bracket $\{x, y\}=<xy, yx, y>$ for all $x, y \in H.$
 \end{proposition}
\begin{corollary}
    In Lemma \ref{0} , \textbf{Aff($\mathfrak{a}$)} is an affgebra which is also Lie affgebra with the bracket $\{f, g\}=<fg, gf, g>$, for all $f, g \in \textbf{Aff($\mathfrak{a}$)}$.
\end{corollary}

\begin{remark}\cite{And2}
    Let $\mathfrak{a}$ be an affine space over $\mathbb{K}.$ Fixed an element $o\in \mathfrak{a},$ let us define a binary operation $+: \mathfrak{a} \times \mathfrak{a} \rightarrow \mathfrak{a}$ and the map $\mathbb{K} \times \mathfrak{a}\rightarrow \mathfrak{a}$ as follows:
    $$a+b:=<a, o, b>, \alpha\cdot a:=\alpha \triangleright_{o}b.$$
    Then, $(\mathfrak{a}, +, \cdot)$ forms a vector space, which is called the tangent space to $\mathfrak{a}$ or the vector
    space fibre of $\mathfrak{a}$ at the point $o$. This tangent space is usually denoted by $T_{o}\mathfrak{a}.$\\

    On the other hand, any vector space $V$ can be understood as an affine space with heap operation $<u, v, w>:=u-v+w$ and the affine action $\alpha\triangleright_{u}v:=(1-\alpha)u+\alpha v$, for all $u, v, w\in V$ and for all $\alpha \in \mathbb{K}.$
\end{remark}
It is proven that any tangent space of a Lie affgebra inherits a natural Lie algebra structure.

\begin{theorem}\cite{And2}
    Let $\mathfrak{a}$ be a Lie affgebra with an affine Lie bracket $\{-, -\}$. Then, for all $o \in \mathfrak{a}$, $T_{o}\mathfrak{a}$ is a Lie algebra with the bracket 
    $$[a, b]:=\{a, b\}-\{a, o\}+\{o, o\}-\{o, b\},$$
    for all $a, b \in T_{o}\mathfrak{a}.$ We call $T_{o}\mathfrak{a}$ with this bracket the Lie algebra tangent to $\mathfrak{a}$ at $o$ or simply a tangent Lie algebra or Lie algebra fible of $\mathfrak{a}.$
\end{theorem}
Let $\mathfrak{g}$ be a Lie algebra. A linear map $\lambda: \mathfrak{g}\rightarrow \mathfrak{g}$ is called a generalized derivation in the sense of Leger and Luks \cite{Leg1}, if there exist $\lambda', \lambda'' \in$ End($\mathfrak{g}$) such that
 $$
 [\lambda(a), b]+[a, \lambda'(b)]=\lambda''([a, b]).
 $$

In the following theorem, the connection between the Lie affgebras and Lie algebras with the generalized derivation is established. The following theorem plays crucial role throughout this paper. 
\begin{theorem}\cite{And2}
    Let $(\mathfrak{g}, [-, -])$ be a Lie algebra and $\kappa, \lambda \in$ End$(\mathfrak{g})$ such that, for all $a, b \in \mathfrak{g}$, 
\begin{equation}\label{eq}
    \lambda([a, b])= [\lambda(a), b]+[a, \lambda(b)]-[a, \kappa(b)]
\end{equation}
Then, for all $a, b , c \in \mathfrak{g}, \alpha \in \mathbb{K}$,  $\mathfrak{g}$ is a Lie affgebra with the affine space structure 
\begin{equation}
    <a, b, c>=a-b+c, ~~\alpha\triangleright_{a}b=(1-\alpha)a+\alpha b,
\end{equation} 
and, with the affine Lie bracket, for all $s\in \mathfrak{g}$, and $a, b \in \mathfrak{g},$ 
\begin{equation}
    \{a, b\}=[a, b]+\kappa(a)+\lambda(b-a)+s.
\end{equation}
 We denote this Lie affgebra by $\mathfrak{a}(\mathfrak{g}; \kappa, \lambda, s).$ Furthermore, for all $o \in \mathfrak{g}$\\
 $$T_{o}\mathfrak{a}(\mathfrak{g}; \kappa, \lambda, s)\simeq \mathfrak{g}.$$\\
 Conversely, for any Lie affgebra $\mathfrak{a}$ and any $o\in \mathfrak{a},$ there exists $\kappa, \lambda, s$ necessarily satisfying $(2)$ and such that $\mathfrak{a}=\mathfrak{a}(T_{o}\mathfrak{a}; \kappa, \lambda, s)$.
\end{theorem}
\begin{definition}
    A affine morphism $\varphi:\mathfrak{a}(\mathfrak{g}; \kappa,\lambda,s)\to \mathfrak{a}(\mathfrak{g'}; \kappa',\lambda',s')$ is said be a homomorphism of Lie affgebras if for all $a,b\in \mathfrak{g}$,
    $$
    \varphi\left(\{a,b\}\right) = \{\varphi(a),\varphi(b)\}.
    $$
\end{definition}

\begin{theorem}\cite{And2}
     A function $\varphi:\mathfrak{a}(\mathfrak{g}; \kappa,\lambda,s)\to \mathfrak{a}(\mathfrak{g'}; \kappa',\lambda',s')$ is a homomorphism of Lie affgebras if and only if there exist a Lie algebra homomorphism
$\psi: \mathfrak{g}\to\mathfrak{g'}$ and $q'\in \mathfrak{g'}$, such that
    \begin{equation}
        \psi \kappa = \kappa'\psi,
    \end{equation}
    \begin{equation}
        \psi \lambda = (ad_{q'} +\lambda')\psi,
    \end{equation}
    \begin{equation}
        \psi (s) = s' - q' +\kappa'(q').
    \end{equation}
\end{theorem}
 \begin{corollary}\cite{And2}
     Let $\mathfrak{a} =\mathfrak{a}(\mathfrak{g}; \kappa,\lambda,s)$ and $\mathfrak{a'} = \mathfrak{a'}(\mathfrak{g'}; \kappa',\lambda',s')$. Then $\mathfrak{a}$ is isomorphic to $\mathfrak{a'}$ if and only if there exist a Lie algebra isomorphism $\Psi: \mathfrak{g}\to \mathfrak{g'}$ and an element $q\in \mathfrak{g}$ such that
        \begin{equation}
            \kappa' = \Psi \kappa \Psi^{-1},
        \end{equation}
        \begin{equation}
            \lambda' = \Psi (\lambda - ad_{q}) \Psi^{-1},
        \end{equation}
        \begin{equation}
            s' = \Psi(s+ q -\kappa(q)).
        \end{equation}
 \end{corollary}
 \begin{definition}
     An affine subspace $\mathfrak{b}$ of a Lie affgebra $\mathfrak{a}$ is called a \textbf{Lie subaffgebra} of $\mathfrak{a}$ if $a\in \mathfrak{b}$, $b\in \mathfrak{b}$ together imply $\{a, b\}\in \mathfrak{b}$.
 \end{definition}
\begin{proposition}\cite{And2}
    Let $\mathfrak{a}=\mathfrak{a}(\mathfrak{g}; \kappa, \lambda, s)$ be a Lie affgebra with affine Lie bracket $\{-, -\}$ and let $\mathfrak{b}$ be an affine subspace of $\mathfrak{a}$. Then $\mathfrak{b}$ is a Lie subaffgebra of $\mathfrak{a}$ if and only if there exists $a\in \mathfrak{b}$ and a Lie subalgebra $\mathfrak{h}$ of $\mathfrak{g}$, such that $\mathfrak{b}=a+\mathfrak{h}$ and 
        \begin{enumerate}
            \item $\kappa(a)+s-a \in \mathfrak{h}$;
            \item $\kappa(\mathfrak{h})\subseteq \mathfrak{h} $ (i.e. $\mathfrak{h}$ is an invariant subspace of $\mathfrak{g}$ with respect to the linear map $\kappa$);
            
            \item $(\lambda + ad_{a}) (\mathfrak{h})\subseteq \mathfrak{h}$ (i.e. $\mathfrak{h}$ is an invariant subspace of $\mathfrak{g}$ with respect to the linear map $(\lambda + ad_{a})$).
        \end{enumerate}
        Furthermore, $\mathfrak{b} \cong \mathfrak{a}(\mathfrak{h}; \kappa, \lambda + ad_{a}, \kappa(a)+s-a)$
        as Lie affgebras.
\end{proposition}

\section{LEFT-IDEAF, RIGHT-IDEAF, IDEAF IN LIE AFFGEBRA}

We introduce \emph{left ideaf}, \emph{right ideaf} and \emph{ideaf} for a \emph{Lie affgebra} with an affine Lie bracket $\{-, - \}$ which is affinization of the concept of ideal in Lie algebra with the Lie bracket $[-, -]$. We give the criterion when a coset of a subspace of a Lie algebra $\mathfrak{g}$ yields a left ideaf, right ideaf and ideaf of $\mathfrak{a}(\mathfrak{g}; \kappa, \lambda, s)$.

\begin{definition}
    An affine subspace $\mathfrak{b}$ of a Lie affgebra $\mathfrak{a}$ is called a \textbf{left ideaf} of $\mathfrak{a}$ if $a\in \mathfrak{a}$, $b\in \mathfrak{b}$ together imply $\{a, b\}\in \mathfrak{b}$.
\end{definition}
\begin{definition}
    An affine subspace $\mathfrak{b}$ of a Lie affgebra $\mathfrak{a}$ is called a \textbf{right ideaf} of $\mathfrak{a}$ if $a\in \mathfrak{a}$, $b\in \mathfrak{b}$ together imply $\{b, a\}\in \mathfrak{b}$.
\end{definition}
\begin{lemma}\label{ideaf}
    Let $\mathfrak{b}$ is an affine subspace of a Lie affgebra $(\mathfrak{a},\{-, -\})$. The following statements are equivalent:
    \begin{enumerate}
        \item If $a \in \mathfrak{a}$, $b \in \mathfrak{b}$ then $\{a, b\} \in \mathfrak{b}$ and $\{a, a\} \in \mathfrak{b}$.
        \item If $a \in \mathfrak{a}$, $b \in \mathfrak{b}$ then $\{a, b\} \in \mathfrak{b}$ and $\{b, a\} \in \mathfrak{b}$.
        \item If $a \in \mathfrak{a}$, $b \in \mathfrak{b}$ then $\{b, a\} \in \mathfrak{b}$ and $\{a, a\} \in \mathfrak{b}$.
    \end{enumerate}
\end{lemma}
    \begin{proof}
        To prove (1) implies (2), we need only to show that $\{b, a\} \in \mathfrak{b}.$ By affine antisymmetry of Lie affgebra, we means that for all $a, b \in \mathfrak{a}$, $<\{a, b\}, \{a, a\}, \{b, a\}>~=\{b, b\}$. By Mal'cev identity, we write $\{a, b\}=~<\{a, b\}, \{b, b\}, \{b, b\}>$, and using affine antisymmetry, we obtain $\{a, b\}=<\{a, a\}, \{b, a\}, \{b, b\}>$, and hence $\{b, a\}=<\{b, b\}, \{a, b\}, \{a, a\}>.$ Since $b\in \mathfrak{b}$, then $\{b, b\}\in \mathfrak{b}$ using $(1)$ and also by $(1)$, $<\{b, b\}, \{a, b\}, \{a, a\}>$ $ \in \mathfrak{b}$, thus $\{b, a\} \in \mathfrak{b}$. \\
        For (2) implies (1), we need only to show that $\{a, a\} \in \mathfrak{b}.$ By affine antisymmetry of Lie affgebra, we obtain $<\{b, a\}, \{b, b\}, \{a, b\}>=\{a, a\}$. Since $b\in \mathfrak{b}$, then by (2), $\{b, b\}\in \mathfrak{b}$ and $<\{b, a\}, \{b, b\}, \{a, b\}>$ $\in \mathfrak{b}$, thus $\{a, a\} \in \mathfrak{b}.$\\
        Similarly, we can argue for the equivalence of (2) and (3) along with (3) and (1).
    \end{proof}
    \begin{definition}
         An affine subspace $\mathfrak{b}$ of a Lie affgebra $\mathfrak{a}$ with affine Lie bracket $\{-, -\}$ is called an \textbf{ideaf} of $\mathfrak{a}$ if it satisfies one of the equivalent conditions of Lemma (\ref{ideaf}).
    \end{definition}
    Given a vector space $V$ and its subspace $W$, any coset $v+W \subseteq V$ is an affine subspace of $V$ in an obvious way:
    $$<v+w, v+w', v+w''>:=v+(w-w'+w''),~~ \alpha \triangleright_{v+w}(v+w'):=v+(1-\alpha)w+ \alpha w',$$ for all $w, w', w'' \in W.$ The criterion when a coset of an (a subspace) ideal of a Lie algebra $\mathfrak{g}$ yields a left ideaf of $\mathfrak{a}(\mathfrak{g}; \kappa, \lambda, s)$ is given in the following
    \begin{prop}\label{1}
        Let $\mathfrak{a}=\mathfrak{a}(\mathfrak{g}; \kappa, \lambda, s)$ be a Lie affgebra with affine Lie bracket $\{-, -\}$ and let $\mathfrak{b}$ be an affine subspace of $\mathfrak{a}$. Then $\mathfrak{b}$ is a left ideaf of $\mathfrak{a}$ if and only if there exists $a\in \mathfrak{b}$ and a ideal $\mathfrak{h}$ of $\mathfrak{g}$, such that $\mathfrak{b}=a+\mathfrak{h}$ and 
        \begin{enumerate}
            \item $\lambda(a)+s-a \in \mathfrak{h}$;
            \item $(\kappa - \lambda - ad_{a})(\mathfrak{g})\subseteq \mathfrak{h};$
            \item $\lambda (\mathfrak{h})\subseteq \mathfrak{h}$ (i.e. $\mathfrak{h}$ is an invariant subspace of $\mathfrak{g}$ with respect to the linear map $\lambda$).
        \end{enumerate}
    \end{prop}
    \begin{proof}
        An affine subspace $\mathfrak{b}$ is of the form $a+\mathfrak{h}$, where $\mathfrak{h}$ is a vector subspace of $\mathfrak{g}$. Then $\mathfrak{b}=a+\mathfrak{h}$ is a left ideaf of $\mathfrak{a}$ if and only if, for all $x\in \mathfrak{a} $ (mainly, of $\mathfrak{g}$), $z \in \mathfrak{b}$ implies $\{x, z\} \in \mathfrak{b}$, since $z\in \mathfrak{b}$ implies $z= a+y,$ for some $y\in \mathfrak{h},$ i.e.,
\begin{align}
\{x, a+y\}&=[x, a+y]+\kappa(x)+\lambda(a+y-x)+s \in a+\mathfrak{h},\tag*{}\\
&=[x, a]+[x, y]+\kappa(x)+\lambda(a)+\lambda(y)-\lambda(x)+s \in a+\mathfrak{h},\tag*{}\\
&=[x, y]+(\kappa- \lambda-ad_{a})(x)+\lambda(a)+\lambda(y)+s \in a+\mathfrak{h}.\label{eq11}
\end{align}
By putting $x=y=0$ in (\ref{eq11}) and using the properties of the Lie bracket $[-, -]$ on $\mathfrak{g}$, we obtain the property $(1)$. In view of $(1)$, by putting $y=0$ in (\ref{eq11}), we get the property $(2)$. In view of the properties $(1)$ and $(2)$, we have the property $(3)$ by putting $x=0$ in (\ref{eq11}). Then together with all the three properties implies that $[x, y] \in \mathfrak{h}$, hence, $\mathfrak{h}$ is a ideal of $\mathfrak{g}$.\\
Conversely, if $\mathfrak{h}$ is ideal of $\mathfrak{g}$ and $(1)$-$(3)$ holds, $\{x, z\}=\{x, a+y\} \in a+ \mathfrak{h},$ so $\mathfrak{b}=a+\mathfrak{h}$ is a left ideaf of $\mathfrak{a}.$
    \end{proof}
    \begin{remark}
        In view of $(2)$, $\mathfrak{h}$ is also an invariant subspace of $\mathfrak{g}$ with respect to the linear map $(\kappa -\lambda-ad_{a}).$
    \end{remark}
    In a similar way, we give criterion when a coset of an (a subspace) ideal of a Lie algebra $\mathfrak{g}$ yields a right Ideal of $\mathfrak{a}(\mathfrak{g}; \kappa, \lambda, s)$ is given in the following
\begin{prop}\label{2}
     Let $\mathfrak{a}=\mathfrak{a}(\mathfrak{g}; \kappa, \lambda, s)$ be a Lie affgebra and let $\mathfrak{b}$ be an affine subspace of $\mathfrak{a}$. Then $\mathfrak{b}$ is a right ideaf of $\mathfrak{a}$ if and only if there exists $a\in \mathfrak{b}$ and a ideal $\mathfrak{h}$ of $\mathfrak{g}$, such that $\mathfrak{b}=a+\mathfrak{h}$ and 
        \begin{enumerate}
            \item $(\kappa - \lambda)(a)+s-a \in \mathfrak{h}$;
            \item $(\lambda + ad_{a})(\mathfrak{g})\subseteq \mathfrak{h};$
            \item $(\kappa -\lambda) (\mathfrak{h})\subseteq \mathfrak{h}$ (i.e. $\mathfrak{h}$ is an invariant subspace of $\mathfrak{g}$ with respect to the linear map $(\kappa-\lambda)$).
        \end{enumerate}
\end{prop}
\begin{proof}
    An affine subspace $\mathfrak{b}$ is of the form $a+\mathfrak{h}$, where $\mathfrak{h}$ is a vector subspace of $\mathfrak{g}$. Then $\mathfrak{b}=a+\mathfrak{h}$ is a right ideaf of $\mathfrak{a}$ if and only if, for all $x\in \mathfrak{a} $, $z \in \mathfrak{b}$, $\{z, x\} \in \mathfrak{b}$. Also, $z\in \mathfrak{b}$ implies $z= a+y,$ for some $y\in \mathfrak{h},$ that is, 
\begin{align} 
    \{a+y, x\}&=[a+y, x]+\kappa(a+y)+\lambda(x-a-y)+s \in a+\mathfrak{h},\nonumber \\ 
              &=[a, x]+[y, x]+\kappa(a)+\kappa(y)+\lambda(x)-\lambda(a)-\lambda(y)+s \in a+\mathfrak{h},\nonumber \\  
               &=[y, x]+(\lambda +ad_{a})(x)+(\kappa- \lambda)(y)+(\kappa-\lambda)(a)+s \in a+\mathfrak{h}.  \label{ideaf eqn}
\end{align}
By putting $x=y=0$ in Equation $\ref{ideaf eqn}$ and using the properties of the Lie bracket $[-, -]$ on $\mathfrak{g}$, we obtain the property $(1)$. In view of $(1)$, by putting $y=0$ in Equation $\ref{ideaf eqn}$, we get the property $(2)$. In view of the properties $(1)$ and $(2)$, we have the property $(3)$ by putting $x=0$ in Equation $\ref{ideaf eqn}$. Then together with all the three properties implies that $[y, x] \in \mathfrak{h}$, hence $\mathfrak{h}$ is a ideal of $\mathfrak{g}$.\\
Conversely, if $\mathfrak{h}$ is an ideal of $\mathfrak{g}$ and $(1)$-$(3)$ holds, $\{z, x\}=\{a+y, x\} \in a+ \mathfrak{h},$ so $\mathfrak{b}=a+\mathfrak{h}$ is a right ideal of $\mathfrak{a}.$
\end{proof}
\begin{remark}
     In view of condition $(2)$ in the above proposition, $\mathfrak{h}$ is also an invariant subspace of $\mathfrak{g}$ with respect to the linear map $(\lambda+ad_{a}).$
\end{remark}
\begin{corollary}
     Let $\mathfrak{a}=\mathfrak{a}(\mathfrak{g}; \kappa, \lambda, s)$ be a Lie affgebra with affine Lie bracket $\{-, -\}$ and let $\mathfrak{b}$ be an affine subspace of $\mathfrak{a}$. Then $\mathfrak{b}$ is an ideaf of $\mathfrak{a}$ if and only if there exists $a\in \mathfrak{b}$ and a ideal $\mathfrak{h}$ of $\mathfrak{a}$ (mainly, of $\mathfrak{g}$), such that $\mathfrak{b}=a+\mathfrak{h}$ and the conditions on $\mathfrak{h}$ in Proposition \ref{1} and Proposition \ref{2} holds.
\end{corollary}
\section{CENTER AND QUOTIENT OF LIE AFFGEBRAS}

In this section, we introduce the notion of center and quotient of Lie affgebras. Firstly, we give the definition of center of Lie affgebras and abelian Lie affgebras. Then, we show that the under certain conditions the center is an ideaf of the given Lie affgebra.

\begin{definition}
    Let $\mathfrak{a}$ be a Lie affgebra with the affine Lie bracket $\{-, -\}$. Then the center of $\mathfrak{a}$  is defined as
    $$\mathfrak{Z(a)}=\{b \in \mathfrak{a}~|~ \{a, b\}=\{b,a\}, \forall a \in \mathfrak{a}\}.$$
\end{definition}
\begin{definition}
    A Lie affgebra $\mathfrak{a}$ is said to be abelian Lie affgebra if for all $a, b \in \mathfrak{a}$, $\{a, b\}=\{b,a\}.$
\end{definition}
\begin{proposition}
Let $\mathfrak{a}=\mathfrak{a}(\mathfrak{g};\kappa, \lambda,s)$ be a Lie affgebra and let $\mathfrak{Z(a)}$ be the center of $\mathfrak{a}$.
    Then $b\in \mathfrak{Z(a)}$ if and only if $b\in Z(\mathfrak{g}$) and whenever $\kappa=2\lambda$.
\end{proposition}
\begin{proof}
Let $b\in \mathfrak{Z(a)}$. Then, by definition of the center, we get $\{a, b\}=\{b,a\}$ for all $a\in \mathfrak{a}$. Now, for any $a, b\in \mathfrak{a}$,
    \begin{align*}
        &~~~~~ \{a,b\}=[a, b]+\kappa(a)+\lambda(b-a)+s,\\
       \text{and}  &~~~~~ \{b,a\}=[b, a]+\kappa(b)+\lambda(a-b)+s.
    \end{align*}
Now, we have
\begin{align*}
&~~~~~~~~~~ \{a, b\}=\{b,a\}\\
&\iff [a, b]+\kappa(a)+\lambda(b-a)+s=[b, a]+\kappa(b)+\lambda(a-b)+s\\
&\iff [a, b]-[b,a]+\kappa(a-b)+2\lambda(b-a)=0, \forall a\in \mathfrak{a},\\
&\iff [a, b]-[b,a]+(\kappa-2\lambda)(a-b)=0, \forall a\in \mathfrak{a},\\
&\iff [a, b]=[b, a], \forall a\in \mathfrak{a},\text{using $\kappa=2\lambda.$}
\end{align*}
\end{proof}

\begin{theorem}
Assume that Char$(\mathbb{K})\neq 2$. Then the center $\mathfrak{Z(a)}$ is an ideaf of the Lie affgebra $\mathfrak{a}(\mathfrak{g};\kappa, \lambda,s)$, whenever $\kappa=2\lambda$ and $[x, \lambda(y)]=0,~ \text{for all}~ x, y\in \mathfrak{a},$ and $\text{ad}_s=0$.
\end{theorem}

\begin{proof}
First we show that $\mathfrak{Z(a)}$ is a sub-heap of $\mathfrak{a}.$ Let $b_{1}, b_{2}, b_{3}\in \mathfrak{Z(a)}$, we need to show that $<b_{1}, b_{2}, b_{3}>~\in \mathfrak{Z(a)}$. Let $a \in \mathfrak{a}$ be arbitrary. Since $\{a, -\}$ is an affine map, then $$\{a, <b_{1}, b_{2}, b_{3}>\}=~<\{a, b_{1}\}, \{a, b_{2}\}, \{a, b_{3}\}>~=~<\{b_{1}, a\}, \{b_{2}, a\}, \{b_{3}, a\}>~=~\{<b_{1}, b_{2}, b_{3}>, a\}.$$ 
Since $b_{1}, b_{2}, b_{3}\in \mathfrak{Z(a)}$, it follows that $<b_{1}, b_{2}, b_{3}> \in \mathfrak{Z(a)}$. Hence, $\mathfrak{Z(a)}$ is a Sub-heap of $\mathfrak{a}.$

 Next, we show that $\mathfrak{Z(a)}$ is an affine subspace of $\mathfrak{a}.$ Let $\alpha \in \mathbb{K}$ and $b_{1}, b_{2} \in \mathfrak{Z(a)}$, we need to show that $\alpha \triangleleft_{b_{1}}b_{2} \in \mathfrak{Z(a)}$. Let $a\in \mathfrak{a}$ be arbitrary. Since $\{a, -\}$ is an affine map, then $\{a, \alpha \triangleleft_{b_{1}}b_{2}\}=\alpha \triangleleft_{\{a, b_{1}\}} \{a, b_{2}\}$.
    Since $b_{1}, b_{2}\in \mathfrak{Z(a)}$, therefore, $\{a, \alpha \triangleleft_{b_{1}}b_{2}\}=\alpha \triangleleft_{\{a, b_{1}\}} \{a, b_{2}\}$=$\alpha\triangleleft_{\{b_{1}, a\}}\{b_2, a\}=\{\alpha\triangleleft_{b_1}b_2,a\}.$\\
    Hence, $\mathfrak{Z(a)}$ is an affine subspace of $\mathfrak{a}.$

Finally, we show $\mathfrak{Z(a)}$ is an ideaf of $\mathfrak{a}$. Let $b\in \mathfrak{Z(a)}$, $c\in \mathfrak{a}$. Then it is sufficient to show that $\{c,b\}\in \mathfrak{Z(a)}$. Let $a\in \mathfrak{a}$ be arbitary element. We need to show that $\{a,\{c,b\}\}=\{\{c, b\}, a\}$. Since $b\in \mathfrak{Z(a)}$, then $\{a, b\}=\{b,a\}$, by definition. Then for all $a, c \in \mathfrak{a}$ and $b\in \mathfrak{Z(a)}$,
\begin{align*}
    &~~~~~~~~~~~\{a,\{c,b\}\}=\{\{c, b\}, a\}\\
    &\iff \{a, \lambda(c)+\lambda(b)+s\}=\{\lambda(c)+\lambda(b)+s, a\}\\
    &\iff [a, \lambda(c)+\lambda(b)+s]+\lambda(a)+\lambda(\lambda(c)+\lambda(b)+s)+s\\
    &=[\lambda(c)+\lambda(b)+s, a]+2\lambda(\lambda(c)+\lambda(b)+s)+\lambda(a-\lambda(c)-\lambda(b)-s)+s\\
    &\iff [a, \lambda(c)]+[a, \lambda(b)]+[a,s]+\lambda(a)+\lambda^2(c)+\lambda^2(b)+\lambda(s)+s\\
    &=[\lambda(c),a]+[\lambda(b),a]+[s,a]+2\lambda^2(c)+2\lambda^2(b)+2\lambda(s)+\lambda(a)-\lambda^2(c)
-\lambda^2(b)-\lambda(s)+s\\
   &\iff 2[a, \lambda(c)]+2[a, \lambda(b)]+2[a, s]=0\\
   &\iff [a, \lambda(c)]+[a, \lambda(b)]+[a, s]=0.
\end{align*}

\end{proof}

\begin{proposition}
   Let $\mathfrak{a}$ be a abelian Lie affgebra with $b\in \mathfrak{a}$ such that $\{a, b\}=\{a, a\}$ for all $a\in \mathfrak{a}$. Then $b\in \mathfrak{a}$ if and only if $\lambda=\text{ad}_b$, for some $b\in \mathfrak{a}$ and $\kappa=0.$ Moreover, for any other $b'(\neq b)\in \mathfrak{a}$, $b-b'\in Z(\mathfrak{g}), \text{center of the underlying Lie algebra}.$
\end{proposition}
\begin{proof}
    Let $b\in \mathfrak{a} = \mathfrak{Z(a)}$, the center of the Lie affgebra $\mathfrak{a}$. Then, we have $\{a, b\}=\{b,a\}=\{a, a\}$, for all $a\in \mathfrak{a}$. Now, for any $a, b\in \mathfrak{a}$,
    \begin{align*}
        &\{a,b\}=[a, b]+\kappa(a)+\lambda(b-a)+s,\\
         &\{b,a\}=[b, a]+\kappa(b)+\lambda(a-b)+s,\\
         &\{a,a\}=\kappa(a)+s.
    \end{align*}
    
From the first two equality, we obtain the following,

\begin{align}
&\{a, b\}=\{b,a\},\nonumber\\
&[a, b]+\kappa(a)+\lambda(b-a)+s=[b, a]+\kappa(b)+\lambda(a-b)+s\nonumber\\
&[a, b]-[b,a]+\kappa(a-b)+2\lambda(b-a)=0.
\end{align}
From the first and third equality, we obtain the following,
\begin{align}
&[a, b]+\kappa(a)+\lambda(b-a)+s=\kappa(a)+s,\nonumber\\
&[a, b]+\lambda(b-a)=0,\nonumber\\ \label{equ13}
&[a, b]=\lambda(a-b).
\end{align}

From the second and third equality, we obtain the following,
\begin{align}
&[b, a]+\kappa(b)+\lambda(a-b)+s=\kappa(a)+s,\nonumber\\
&[b,a]+(\kappa-\lambda)(b-a)=0,\nonumber \\\label{equ14}
&[b,a]=(\kappa-\lambda)(a-b).
\end{align}
We know that in the Lie algebra $\mathfrak{g}$, by anti-commutativity, $[a, b]=-[b,a]$ for all $a, b\in \mathfrak{g}.$ Thus, using Equation $\ref{equ13}$ and $\ref{equ14}$, we get 
\begin{align}
&\lambda(a-b)=-(\kappa-\lambda)(a-b),\nonumber\\
&\kappa(a-b)=0, ~\text{for all}~ a,b \in \mathfrak{g},\nonumber\\ \label{equ15}
&\kappa(a)=\kappa(b),~\text{for all}~ a,b \in \mathfrak{g}.
\end{align}
Choosing $a=o\in \mathfrak{g},$ which is the additive identity of the Lie algebra $\mathfrak{g}$, we have $ k(b)=o$. Therefore, using Equation \ref{equ15}, we have $k(a)=0,$ for all $a\in \mathfrak{g}$. Thus, $\kappa=0.$
Using Equation $\ref{equ13}$, for all $a$, $[a, b]=\lambda(a-b)$. Again, we choose $a=o\in \mathfrak{g}$. Then, $[o, b]=o=-\lambda(b)$ and hence, $[a, b]=\lambda(a)$. Similarly, Using equation $\ref{equ14}$, for all $a$, $[b, a]=\lambda(b-a)$. Again, we choose $a=o\in \mathfrak{g}.$ Then, $[b, o]=o=\lambda(b)$ and hence $[b, a]=-\lambda(a)$. Therefore, $\text{ad}_b=\lambda.$\\
Conversely, let us assume that $\kappa=0$ and  $\lambda=\text{ad}_b$, for some $b\in \mathfrak{a}$.Then  $\{a, b\}=[a,b]+\kappa(a)+\lambda(b-a)+s=[a,b]+\text{ad}_b(b-a)+s=[a,b]+[b-a, b]+s=[a, b]+[b, b]-[a, b]+s=s$, $\{b, a\}=[b,a]+\kappa(b)+\lambda(a-b)+s=[b,a]+\lambda(a-b)+s=[b,a]+\text{ad}_b(a-b)+s=[b,a]+[a-b, b]+s=[b,a]+[a, b]-[b,b]+s=s$ and $\{a, a\}=s,$ for all $a\in \mathfrak{a}.$ Then for all $a\in \mathfrak{a}$, $\{a, b\}=\{b,a\}=\{a, a\}.$\\
Let $b'(\neq b) \in \mathfrak{Z(a)}$, then by above observation, $\text{ad}_b(a)=\text{ad}_{b'}(a)=\lambda(a)$. It follows that $[a, b]=[a, b']$ and then $[a, b-b']=0, \forall a \in\mathfrak{a}$. Thus, $b-b'\in Z(\mathfrak{g}).$
    
\end{proof}

\subsection{QUOTIENT LIE AFFGEBRAS}
We start with the quotient heap as in \cite{Brz22} and quotient affine space as in \cite{Brz22}. Then we define quotient Lie affgebras in an obvious way.
Let $\mathfrak{a}$ be a Lie affgebra and $\mathfrak{b}$ be an ideaf of $\mathfrak{a}$. We start by assigning a relation to a sub-heap $\mathfrak{b}$ of the Lie affgebras $\mathfrak{a}$.
\begin{definition}\cite{Brz22}
    Given a sub-heap $\mathfrak{b}$ of $\mathfrak{a}$, we define a sub-heap relation $\sim_{\mathfrak{b}}$ as follows: $a_1\sim_{\mathfrak{b}} a_2 \iff$ there exists $b\in \mathfrak{b}$ such that $<a_1, a_2, b>$ $ \in \mathfrak{b}$.
\end{definition}
\begin{proposition}\cite{Brz22}
    Let $\mathfrak{b}$ be a sub-heap of $\mathfrak{a}$.
    \begin{enumerate}
        \item The relation $\sim_{\mathfrak{b}}$ is an equivalence relation.
        \item For all $b\in \mathfrak{b}$, class of $b$ is equal to $\mathfrak{b}$.
        \item For all $a_1, a_2 \in \mathfrak{a}$, $a_1 \sim_{\mathfrak{b}} a_2 $if and only if for all $b\in \mathfrak{b}$, $<a_1, a_2, b>$ $\in \mathfrak{b}.$
        \item If $\mathfrak{b}$ is a normal sub-heap of $\mathfrak{a}$, then the set of equivalence classes $\mathfrak{a}/\mathfrak{b}$ is a heap with the inherited operation: $<\overline{a_1}, \overline{a_2},\overline{a_3} >=\overline{<a_1, a_2, a_3>}$, where $\overline{a}\in \mathfrak{a}/\mathfrak{b}$ is the class of $a \in \mathfrak{a}.$ 
    \end{enumerate}
\end{proposition}

\begin{remark}
    For $a_1\in \mathfrak{a}$, $\overline{a_1}=\{a_2\in \mathfrak{a}~| <a_1, a_2, b>~\in \mathfrak{b},  \forall b\in \mathfrak{b} \}$ and then we write $<a_1, a_2, b>~=b'$ for some $b' \in \mathfrak{b}$. Then $a_2=<a_1, <b, a_2, a_1>, b>=<a_1,<a_1, a_2,b>, b>~=<a_1, b', b>,$ by using Lemma$(2.3)(3)$ in \cite{Brz22}. Thus, $\overline{a_1}=\{ <a_1, b', b>| \forall b\in \mathfrak{a}, \exists b'\in \mathfrak{a}\}$.
\end{remark}
The set of equivalence classes $\mathfrak{a}/\mathfrak{b}$ is an affine space with 
$<\overline{a_1}, \overline{a_2},\overline{a_3} >=\overline{<a_1, a_2, a_3>} $ and $\alpha\triangleleft_{\overline{a_1}}\overline{a_2}=\overline{\alpha\triangleleft_{a_1}a_2}$. The ternary action is well defined, see [\cite{Brz1}, Construction $4.5$].\\

The set of equivalence classes $\mathfrak{a}/\mathfrak{b}$ is a Lie affgebra with ternary operation $<\overline{a_1}, \overline{a_2},\overline{a_3} >:=\overline{<a_1, a_2, a_3>} $ and ternary action $\alpha\triangleleft_{\overline{a_1}}\overline{a_2}:=\overline{\alpha\triangleleft_{a_1}a_2}$, and with affine Lie bracket $\{\overline{a_1}, \overline{a_2}\}:=\overline{\{a_1, a_2\}}$, where $\mathfrak{b}$ is an ideaf of Lie affgebra $\mathfrak{a}.$

Now, we show that the affine Lie bracket is well-defined. Observe\\
 $ \{<a_1, b_1, b_2>, <a_2, b_3, b_4>\}\\
=~<<\{a_1, a_2\},\{a_1, b_3\},\{a_1, b_4\}>, \underbrace{<\{b_1, a_2\}, \{b_1, b_3\}, \{b_1, b_4\}>}_{=b_5\in \mathfrak{b}},\underbrace{<\{b_2, a_2\}, \{b_2, b_3\}, \{b_2, b_4\}>}_{=b_6\in \mathfrak{b}}>$ (Since $\mathfrak{a}$ is a Lie affgebra and $\mathfrak{b}$ is ideaf of $\mathfrak{a}$)\\
$=~<<\{a_1, a_2\},\{a_1, b_3\},\{a_1, b_4\}>, b_5, b_6>$\\
$=~<\{a_1, a_2\},\underbrace{\{a_1, b_3\}}_{=b_8\in \mathfrak{b}},\underbrace{<\{a_1, b_4\}, b_5, b_6>}_{=b_7\in \mathfrak{b}}>$
(Since $<-,-,->$ is associative)\\
$=~<\{a_1, a_2\}, b_8, b_7>$.\\
This shows that the affine Lie bracket on $\mathfrak{a}/\mathfrak{b}$ is well defined.

\begin{example}
    Let $\mathfrak{a}$ be a Lie affgebra, and let $\mathfrak{Z}(\mathfrak{a})$ denote its center. If $\mathfrak{Z}(\mathfrak{a})$ satisfies the necessary conditions to be an ideaf, then the quotient $\mathfrak{a}/\mathfrak{Z}(\mathfrak{a})$ is a Lie affgebra.
\end{example}

\section{PRODUCT OF IDEAFS}

In this section, we define \emph{product ideaf} $\{\mathfrak{a}, \mathfrak{b}\}$ for a \emph{Lie affgebra} with an affine Lie bracket $\{-, -\}.$ We also give criterion of \emph{product ideaf} $\{\mathfrak{a}, \mathfrak{b}\}$ in view of the criterion of \emph{ideafs} $\mathfrak{a}, \mathfrak{b}$ of Lie affgebra.

Let $\mathfrak{c}(\mathfrak{g}; \kappa, \lambda, s)$ be a Lie affgebra. Let $\mathfrak{a}( =a + \mathfrak{U})$, $\mathfrak{b}( =b + \mathfrak{V})$ be two (left as well as right) ideafs of $\mathfrak{c}$. Let us define $\mathfrak{I}=\{\mathfrak{a}, \mathfrak{b}\}:=$ smallest affine subspace containing $\{x, y\}$ where $x\in \mathfrak{a}, y \in \mathfrak{b}$. Then $\mathfrak{J}$ is a left ideaf, right ideaf and ideaf of $\mathfrak{c}(\mathfrak{g}; \kappa, \lambda, s)$ with certain conditions on the associated vector spaces $\mathfrak{U}$, $\mathfrak{V},$ $\mathfrak{U}+\mathfrak{V}$ or the product ideal $[\mathfrak{U}, \mathfrak{V}]$ of $\mathfrak{g}$.

The affine subspace $\mathfrak{a}$ is necessarily of the form $\mathfrak{a}=a + \mathfrak{U}$, where $\mathfrak{U}$ is an ideal of $\mathfrak{g}$ and the affine subspace $\mathfrak{b}$ is necessarily of the form $\mathfrak{b}=b + \mathfrak{V}$, where $\mathfrak{V}$ is an ideal of $\mathfrak{g}$. We would like to investigate the underlying affine space for which $\mathfrak{I}=\{\mathfrak{a}, \mathfrak{b}\}$ is a left ideaf of $\mathfrak{c}(\mathfrak{g}; \kappa, \lambda, s)$.

First of all, we calculate for left ideaf. Let $g \in \mathfrak{c}(\mathfrak{g}; \kappa, \lambda, s)$ and $\sum_{i}\lambda_{i}\{a+u_{i}, b+v_{i}\} \in \{\mathfrak{a}, \mathfrak{b}\}$, where $\sum_{i}\lambda_{i}=1$, $u_{i}\in \mathfrak{U}$ and $v_{i}\in \mathfrak{V}.$ Observe that

\begin{align}
\{g, \sum_{i}\lambda_{i}\{a+u_{i}, b+v_{i}\}\}&=[g, \sum_{i}\lambda_{i}\{a+u_{i}, b+v_{i}\}]+\kappa(g)+\lambda(\sum_{i}\lambda_{i}\{a+u_{i}, b+v_{i}\}-g)+s\tag*{}\\
&=\underbrace{\sum_{i}\lambda_{i} [g, \{a+u_{i}, b+v_{i}\}]}_{1^{st} }+(\kappa-\lambda)(g)+s+\underbrace{\sum_{i}\lambda_{i} \lambda(\{a+u_{i}, b+v_{i}\})}_{4^{th}}. \label{eq14}
\end{align}
First of all, we have
\begin{align}
\{a+u_{i}, b+v_{i}\}&=[a+u_{i},b+v_{i}]+\kappa(a+u_{i})+\lambda(b-a+v_{i}-u_{i})+s\tag*{}\\
&=[a, b]+[a, v_{i}]+[u_{i}, b]+[u_{i}, v_{i}]+(\kappa-\lambda)(u_{i})+\lambda(v_{i})+(\kappa-\lambda)(a)+\lambda(b)+s\tag*{}\\
&=[a, b]+(\lambda+ad_{a})(v_{i})+(\kappa-\lambda-ad_{b})(u_{i})+[u_{i}, v_{i}]+(\kappa-\lambda)(a)+\lambda(b)+s. \label{eq15}
\end{align}
Using (\ref{eq15}) in (\ref{eq14}), $1^{st}$ term implies 
\begin{align}
&\sum_{i}\lambda_{i}[g, [a, b]+(\lambda+ad_{a})(v_{i})+(\kappa-\lambda-ad_{b})(u_{i})+[u_{i}, v_{i}]+(\kappa-\lambda)(a)+\lambda(b)+s]\tag*{}\\
=&[g, [a, b]+(\kappa-\lambda)(a)+\lambda(b)+s]+\sum_{i}\lambda_{i}[g, (\lambda+ad_{a})(v_{i})+(\kappa-\lambda-ad_{b})(u_{i})+[u_{i}, v_{i}]],~ \text{since $\sum_{i}\lambda_{i}=1$.}\\
=&[g, \{a, b\}]+\sum_{i}\lambda_{i}[g, \{u_{i}, v_{i}\}-s+ad_{a}(v_{i})-ad_{b}(u_{i})].
\end{align}

Using (\ref{eq15}) in (\ref{eq14}), $4^{th}$ term implies 
\begin{align}
&\sum_{i}\lambda_{i}\lambda([a, b]+(\lambda+ad_{a})(v_{i})+(\kappa-\lambda-ad_{b})(u_{i})+[u_{i}, v_{i}]+(\kappa-\lambda)(a)+\lambda(b)+s)\tag*{}\\
=&\lambda([a, b]+\lambda(b)+s+(\kappa-\lambda)(a))+\sum_{i}\lambda_{i}\lambda([u_{i}, v_{i}]+(\lambda+ad_{a})(v_{i})+(\kappa-\lambda-ad_{b})(u_{i})),~ \text{since $\sum_{i}\lambda_{i}=1.$}\nonumber \\
=&\lambda(\{a, b\})+\sum_{i}\lambda_{i}\lambda(\{u_{i}, v_{i}\}-s+ad_{a}(v_{i})-ad_{b}(u_{i})).
\end{align}
Since $\mathfrak{U},$ $ \mathfrak{V}$ is an ideal of $\mathfrak{g}$ then $\mathfrak{U}+\mathfrak{V}$ is also an ideal of $\mathfrak{g}.$ Since $\mathfrak{a}=a+\mathfrak{U}$, $\mathfrak{b}=b+\mathfrak{V}$ both are ideafs of $\mathfrak{c}$, so using those above conditions on ideals, we get the membership relation. Putting all of them together into (\ref{eq14}), we obtain the following
\begin{align}
&=(\kappa-\lambda)(g)+s+\underbrace{[g, [a, b]+(\kappa-\lambda)(a)+\lambda(b)+s]}_{(1)}+\underbrace{\sum_{i}\lambda_{i}[g, \underbrace{\underbrace{(\lambda+ad_{a})(v_{i})}_{\in \mathfrak{U}}+\underbrace{(\kappa-\lambda-ad_{b})(u_{i})+[u_{i}, v_{i}]}_{\in \mathfrak{V}}]}_{\in \mathfrak{U}+\mathfrak{V}}}_{\in \mathfrak{U}+\mathfrak{V}}+\tag*{}\\
&\underbrace{\lambda([a, b]+\lambda(b)+s+(\kappa-\lambda)(a))}_{(\in \mathfrak{U}, \in \mathfrak{U}+\mathfrak{V}) }+\underbrace{\sum_{i}\lambda_{i}\lambda(\underbrace{\underbrace{[u_{i}, v_{i}]+(\lambda+ad_{a})(v_{i})}_{\in \mathfrak{U}}+\underbrace{(\kappa-\lambda-ad_{b})(u_{i})}_{\in \mathfrak{V}}}_{\in \mathfrak{U}+\mathfrak{V}})}_{\in \mathfrak{U}+\mathfrak{V}}.\label{eq20}
\end{align}
From Equation (\ref{eq20}), 
\begin{align*}
(1)=[g, (\kappa-\lambda)(a)+s-a+(ad_{a}+\lambda)(b)+a]=\underbrace{\underbrace{[g, \underbrace{(\kappa-\lambda)(a)+s-a}_{\in \mathfrak{U}}]}_{\in \mathfrak{U}}+\underbrace{[g, \underbrace{(ad_{a}+\lambda)(b)}_{\in \mathfrak{U}}]}_{\in \mathfrak{U}}}_{\in \mathfrak{U}, \in \mathfrak{U}+\mathfrak{V} }+[g, a].
\end{align*}
Thus, (\ref{eq20}) becomes 
\begin{align*}
&(\kappa-\lambda)(g)+s+[g, a]+\text{ some elements of } \mathfrak{U}+\mathfrak{V}\\
=&\underbrace{(\kappa-\lambda-ad_{a})(g)}_{\in \mathfrak{U}, \in \mathfrak{U}+\mathfrak{V}}+s+ \text{some elements of } \mathfrak{U}+\mathfrak{V}\\
=&s+\text{ some elements of } \mathfrak{U}+\mathfrak{V}.
\end{align*}
Let us consider that the affine space is $s+ (\mathfrak{U}+\mathfrak{V})$ then $\{\mathfrak{a}, \mathfrak{b}\}$ is a left ideaf.

Now, we calculate for the right ideaf.

Similar calculations is applied for the right ideaf of $\{\mathfrak{a}, \mathfrak{b}\}$.
Let $g \in \mathfrak{c}(\mathfrak{g}; \kappa, \lambda, s)$ and $\sum_{i}\lambda_{i}\{a+u_{i}, b+v_{i}\} \in \{\mathfrak{a}, \mathfrak{b}\}$ where $\sum_{i}\lambda_{i}=1$, $u_{i}\in \mathfrak{U}$ and $v_{i}\in \mathfrak{V}.$ We would like to investigate the underlying affine space for which $\mathfrak{I}=\{\mathfrak{a}, \mathfrak{b}\}$ is a right ideaf of $\mathfrak{c}(\mathfrak{g}; \kappa, \lambda, s)$.
\begin{align}
\therefore \{\sum_{i}\lambda_{i}\{a+u_{i}, b+v_{i}\}, g\}&=[\sum_{i}\lambda_{i}\{a+u_{i}, b+v_{i}\}, g]+\kappa(\sum_{i}\lambda_{i}\{a+u_{i}, b+v_{i}\})+\lambda(g-\sum_{i}\lambda_{i}\{a+u_{i}, b+v_{i}\})+s\tag*{}\\
&=\sum_{i}\lambda_{i}[\{a+u_{i}, b+v_{i}\},g]+\sum_{i}\lambda_{i}\kappa(\{a+u_{i}, b+v_{i}\})-\sum_{i}\lambda_{i}\lambda(\{a+u_{i}, b+v_{i}\})+\lambda(g)+s\tag*{}\\
&=\lambda(g)+s+\underbrace{\sum_{i}\lambda_{i}[\{a+u_{i}, b+v_{i}\},g]}_{3^{rd} term}+\underbrace{\sum_{i}\lambda_{i}(\kappa-\lambda)(\{a+u_{i}, b+v_{i}\})}_{4^{th} term}. \label{eq21}
\end{align}
Using (\ref{eq15}) in (\ref{eq21}), $3^{rd}$ term implies 
\begin{align*}
&\sum_{i}\lambda_{i}[\{a+u_{i}, b+v_{i}\},g]=\sum_{i}\lambda_{i}[[a, b]+(\lambda+ad_{a})(v_{i})+(\kappa-\lambda-ad_{b})(u_{i})+[u_{i}, v_{i}]+(\kappa-\lambda)(a)+\lambda(b)+s,g]\\
=&\sum_{i}\lambda_{i}[[a,b]+(\kappa-\lambda)(a)+\lambda(b)+s, g]+\sum_{i}\lambda_{i}[(\lambda+ad_{a})(v_{i}),g]+\sum_{i}\lambda_{i}[(\kappa-\lambda-ad_{b})(u_{i}),g]+\sum_{i}\lambda_{i}[[u_{i}, v_{i}],g]\\
=&[[a,b]+(\kappa-\lambda)(a)+\lambda(b)+s, g]+\underbrace{\underbrace{\sum_{i}\lambda_{i}\underbrace{[\underbrace{(\lambda+ad_{a})(v_{i})}_{\in \mathfrak{U}},g]}_{\in \mathfrak{U}}}_{\in \mathfrak{U}}+\underbrace{\sum_{i}\lambda_{i}\underbrace{[\underbrace{(\kappa-\lambda-ad_{b})(u_{i})}_{\in \mathfrak{V}},g]}_{\in \mathfrak{V}}}_{\in \mathfrak{V}}+\underbrace{\sum_{i}\lambda_{i}\underbrace{[\underbrace{[u_{i}, v_{i}]}_{\in \mathfrak{V}},g]}_{\in \mathfrak{V}}}_{\in \mathfrak{V}}}_{\in \mathfrak{U+V}}\\
=&[(\lambda+ad_{a})(b)+(\kappa-\lambda)(a)+s-a+a,g] + \text{some elements of $\mathfrak{U+V}$}\\
=&[a,g]+\underbrace{[\underbrace{\underbrace{(\lambda+ad_{a})(b)}_{\in \mathfrak{U}}+\underbrace{(\kappa-\lambda)(a)+s-a}_{\in \mathfrak{\mathfrak{U}}}}_{\in \mathfrak{U}}, g]}_{\in \mathfrak{U}/\in \mathfrak{U+V}}+ \text{some elements of $\mathfrak{U+V}$}\\
=&[a,g]+\text{ some elements of $\mathfrak{U+V}$}.
\end{align*}

Using (\ref{eq15}) in (\ref{eq21}), $4^{th}$ term implies
\begin{align*}
&\sum_{i}\lambda_{i}(\kappa-\lambda)(\{a+u_{i}, b+v_{i}\})=\sum_{i}\lambda_{i}(\kappa-\lambda)([a, b]+(\lambda+ad_{a})(v_{i})+(\kappa-\lambda-ad_{b})(u_{i})+[u_{i}, v_{i}]+(\kappa-\lambda)(a)+\lambda(b)+s)\\
=&\sum_{i}\lambda_{i}(\kappa-\lambda)([a, b]+\underbrace{(\lambda+ad_{a})(v_{i})}_{\in \mathfrak{U}}+\underbrace{(\kappa-\lambda-ad_{b})(u_{i})}_{\in \mathfrak{V}}+\underbrace{[u_{i}, v_{i}]}_{\in \mathfrak{V}}+(\kappa-\lambda)(a)+\lambda(b)+s)\\
=&\sum_{i}\lambda_{i}(\kappa-\lambda)(\underbrace{(\lambda+ad_{a})(v_{i})}_{\in \mathfrak{U}}+\underbrace{(\kappa-\lambda-ad_{b})(u_{i})}_{\in \mathfrak{V}}+\underbrace{[u_{i}, v_{i}]}_{\in \mathfrak{V}}+[a,b]+(\kappa-\lambda)(a)+\lambda(b)+s)\\
=&\sum_{i}\lambda_{i}(\kappa-\lambda)(\underbrace{(\lambda+ad_{a})(v_{i})}_{\in \mathfrak{U}}+\underbrace{(\kappa-\lambda-ad_{b})(u_{i})}_{\in \mathfrak{V}}+\underbrace{[u_{i}, v_{i}]}_{\in \mathfrak{V}}+\underbrace{(\lambda+ad_{a})(b)}_{\in \mathfrak{U}}+\underbrace{(\kappa-\lambda)(a)+s-a}_{\in \mathfrak{U}}+a)\\
=&\sum_{i}\lambda_{i}(\kappa-\lambda)(a)+\underbrace{\sum_{i}\lambda_{i}\underbrace{(\kappa-\lambda)\underbrace{(\underbrace{(\kappa-\lambda-ad_{b})(u_{i})}_{\in \mathfrak{V}}+\underbrace{[u_{i}, v_{i}]}_{\in \mathfrak{V}})}_{\in \mathfrak{V}}}_{\in \mathfrak{V}}}_{\in \mathfrak{V}}+\\
&\underbrace{\sum_{i}\lambda_{i}\underbrace{(\kappa-\lambda)\underbrace{(\underbrace{(\lambda+ad_{a})(v_{i})}_{\in \mathfrak{U}}+\underbrace{(\lambda+ad_{a})(b)}_{\in \mathfrak{U}}+\underbrace{(\kappa-\lambda)(a)+s-a}_{\in \mathfrak{U}}))}_{\in \mathfrak{U}}}_{\in \mathfrak{U}}}_{\in \mathfrak{U}}
\end{align*}
Since $\sum_{i}\lambda_{i}=1,$ $4^{th}$ term $=(\kappa-\lambda)(a)+$ some of elements of $\mathfrak{U+V}$.\\

Thus,  using $3^{rd}$ term and $4^{th}$ term (\ref{eq21}) becomes
\begin{align*}
&\lambda(g)+s+[a,g]+(\kappa-\lambda)(a)+\text{ some of elements of $\mathfrak{U+V}$}\\
=&(\lambda+ad_{a})(g)+(\kappa-\lambda)(a)+s+\text{some of elements of $\mathfrak{U+V}$}.
\end{align*}
Since $ad_{a}(a)=0$, (\ref{eq21})$=\underbrace{(\lambda+ad_{a})(g)}_{\in \mathfrak{U}/\in \mathfrak{U+V}}+\underbrace{(\kappa-\lambda-ad_{a})(a)}_{\in \mathfrak{U}/\in \mathfrak{U+V}}+s+$ some of elements of $\mathfrak{U+V}$.

Let us consider the affine space $s+(\mathfrak{U+V})$, then $\{\mathfrak{a}, \mathfrak{b}\}$ is a right ideaf.

Thus, $\{\mathfrak{a}, \mathfrak{b}\}$ is an ideaf of $\mathfrak{c}(\mathfrak{g}; \kappa, \lambda, s)$, whose underlying affine space is $s+(\mathfrak{U+V})$.

\section{CONCEPT OF AFFINE NILPOTENCY IN LIE AFFGEBRA}

In this section, we introduce the concept of nilpotency of Lie affgebra $\mathfrak{a}(\mathfrak{g}; \kappa, \lambda, s)$.

Let $\mathfrak{a}$ be a Lie affgebra with affine Lie bracket $\{-,-\}$. Let us consider a subset $E$ of $\mathfrak{a}$ defined by 
\begin{equation}\label{set E}
E:=\{ x\in \mathfrak{a}~|~ \{x,a\}=\{a, x\}=x, \forall~ a\in \mathfrak{a}\}.
\end{equation}
Note that either $E$ is empty set or $E$ is a singleton set (why?). If $E\neq \emptyset$, choose $x, y\in E$, then by the definition of $E,$ we conclude that $x=y.$\\
Recall, for all $o$ in $\mathfrak{a}$, $T_o\mathfrak{a}$ is a Lie algebra with the bracket $$[a, b]:=\{a,b\}-\{a,o\}+\{o,o\}-\{o,b\},$$
for all $a. b\in T_o\mathfrak{a}.$ Then, we may consider the following subset of $T_o\mathfrak{a}$
\begin{equation}\label{set E0}
E_o:=\{ x\in T_o\mathfrak{a}~|~ [x, a]=[a, x]=x, \forall~ a\in T_o\mathfrak{a}\}.
\end{equation}
We show that $E_o=\{o\}.$\\
First of all, we show that $o\in E_o$. Let $a\in T_o\mathfrak{a}$ be an arbitrary element. Then $[o,a]=\{o,a\}-\{o,o\}+\{o,o\}-\{o,a\}=o$ and also $[a,o]=\{a,o\}-\{a,o\}+\{o,o\}-\{o,o\}=o.$ Thus, $[o,a]=[a,o]=o$, $\forall~ a\in T_o\mathfrak{a}$, i.e. $o\in E_o.$\\ If possible, let $x(\neq o)\in E_o$. Then $[x, a]=[a,x]=x, \forall~ a\in T_o\mathfrak{a}$. Since the last equation holds for all $a\in T_o\mathfrak{a}$, in particular, we may choose $a=x$, then we get $$[x,x]=x, \quad [x,x]=x. $$
This implies \begin{align*}
 x&=[x,x]\\
 &=\{x,x\}-\{x,o\}+\{o,o\}-\{o,x\}\\
 &=\{x,x\}-(\{x,o\}-\{o,o\}+\{o,x\})\\
 &=\{x,x\}-\{x,x\}\\
 &=o.
\end{align*}
This leads to a contradiction of the assumption on our choice of $x.$ Hence, our claim is verified. 

\begin{definition}
    A Lie affgebra $(\mathfrak{a},\{-,-\})$ is said to be affinely nilpotent Lie affgebra if for the sequence of (product) ideafs (left as well as right ideafs) $(\mathfrak{a}^{(m)})_{m \in \mathbb{Z}_{\geq 0}}$ of $\mathfrak{a}(\mathfrak{g}; \kappa, \lambda, s)$  defined by $\mathfrak{a}^{(0)}:=\mathfrak{a}$, $\mathfrak{a}^{(1)}:=\{\mathfrak{a}, \mathfrak{a}\}$, $\mathfrak{a}^{(2)}:=\{\mathfrak{a}, \mathfrak{a}^{(1)}\}$, $\ldots,$~ $ \mathfrak{a}^{(m)}:=\{\mathfrak{a}, \mathfrak{a}^{(m-1)}\}, \ldots $, there exists a positive integer $n$ such that $\mathfrak{a}^{(n+k)}=\mathfrak{a}^{(n)}=E$ for all $k\in \mathbb{N}$.
\end{definition}
The $n$ in the above definition is called degree/index of the nilpotency.

Let $T_o\mathfrak{a}$ be a Lie algebra induced from the Lie affgebra $\mathfrak{a}$ by retracting to the point $o.$ Then the definition of affine nilpotency in the sense of the Lie affgebra $\mathfrak{a}$ leads to the definition of nilpotency in sense of the Lie algebra $T_o\mathfrak{a}$. Thus, $$(T_o\mathfrak{a})^{(n)}=(T_o\mathfrak{a})^{(n+k)}=E_o=\{o\}, \forall~ k\geq 1.$$
This shows that $T_o\mathfrak{a}$ is a nilpotent Lie algebra of index $n.$
\begin{theorem}\label{affgebra nilpotent to algebra nilpotent}
    Let $(\mathfrak{a}, \{-, -\})$ be an affinely nilpotent Lie affgebra of index $n$, and $o$ be any element of $\mathfrak{a}$. Then $(T_o\mathfrak{a}, [-,-])$ is a nilpotent Lie algebra of same index.
    
\end{theorem}
\begin{proof}
For $n=0$, the proof is trivial. For $n=1$, $\mathfrak{a}^{(1)}=\{\mathfrak{a}, \mathfrak{a}\}=E=\{e\}.$ Since $E$ is a singleton set, we may assume $E=\{e\}$. We prove that $T_o\mathfrak{a}$ is a nilpotent Lie algebra of index $1.$ Since $(T_o\mathfrak{a})^{(1)}=[T_o\mathfrak{a}, T_o\mathfrak{a}].$ We need to show that $[T_o\mathfrak{a}, T_o\mathfrak{a}]=E_o=\lbrace o \rbrace$. Now, $[a, b]:=\{a, b\}-\{a,o\}+\{o,o\}-\{o,b\}=e-e+e-e= o.$ Thus, $T_o\mathfrak{a}$ is a nilpotent Lie algebra of index $1.$ For $n=2,$ $\mathfrak{a}^{(2)}=\{\mathfrak{a}, \mathfrak{a}^{(1)}\}=E=\{e\}$,~i.e., $\{a, \{b, c\}\}=e, \forall~ a, b,c \in \mathfrak{a}.$  We need to show that $[T_o\mathfrak{a}, [T_o\mathfrak{a}, T_o\mathfrak{a}]]=E_o=o$. For any $a, b, c\in T_o\mathfrak{a}$, 
 \begin{align*}
      \quad \quad [a,[b, c]] &\ =[a, \{b, c\}-\{b, o\}+\{o,o\}-\{o,  \,c\}]\\
     &\ =  [a, \{b,c\}]-[a, \{b,o\}]+[a, \{o,o\}]-[a, \{o,c\}]\\
     &\ = \{a,\{b, c\}\}- \{a, o\}+\{o, o\}-\{o, \{b, c\}\}\\
     &\ \quad -\{a,\{b, o\}\}+ \{a, o\}-\{o, o\}+\{o, \{b, o\}\}\\
     &\ \quad +\{a,\{o, o\}\}- \{a, o\}+\{o, o\}-\{o, \{o, o\}\}\\
     &\ \quad -\{a,\{o, c\}\}+ \{a, o\}-\{o, o\}+\{o, \{o, c\}\}\\
     &\ = e-e-e+e+e-e-e+e\\
     &\ =o.
 \end{align*}
 Hence, $T_o\mathfrak{a}$ is a nilpotent Lie algebra of index $2.$ For general index $n$, similar computation will be carried out, i.e., $[a_1, [a_2, [\cdots a_{n+1}\underbrace{]]]\cdots ]}_{n \text{-terms }}=$ Sum of $2^{n+1}$ elements with alternating sign of $\{x_1, \{x_2, \{\cdots x_{n+1}\underbrace{\}\}\}\cdots \}}_{n \text{-terms }}=e,$ for some $x_i\in \mathfrak{a}$. Thus, $[a_1, [a_2, [\cdots a_{n+1}\underbrace{]]]\cdots ]}_{n \text{-terms }}=o.$ Hence, $T_o\mathfrak{a}$ is a nilpotent Lie algebra of index $n.$
\end{proof}


\begin{theorem}
    Let $\mathfrak{a}=(T_o\mathfrak{a}, 0, 0; s)$ is an affinely nilpotent Lie affgebra if and only if $T_o\mathfrak{a}$ is a nilpotent Lie algebra, whenever $\text{ad}_s=0.$
\end{theorem}
\begin{proof}
    Implication part is follows from the Theorem \ref{affgebra nilpotent to algebra nilpotent}.
    Let $T_o\mathfrak{a}$ is a nilpotent Lie algebra of index $n$ and $\text{ad}_s=0.$ 
    For $a_1, a_2, \cdots, a_n \in \mathfrak{a}$, 
    
\begin{align*}
    &\{a_1, \{a_2,\{a_3,\cdots \{a_{n-1}, a_n\}\}\cdots\},\\
    =&\{a_1, \{a_2,\{a_3,\cdots, \{a_{n-2} ,[a_{n-1}, a_n]+s\}\cdots\},\\
    =&\{a_1, \{a_2,\{a_3,\cdots, \{a_{n-3}, [a_{n-2} ,[a_{n-1}, a_n]]+\underbrace{[a_{n-2},s]}_{\text{=0, as ad}_s=0}+s\}\cdots\},\\
    =& \{a_1, \{a_2,\{a_3,\cdots,\{a_{n-4}, [a_{n-3}, [a_{n-2} ,[a_{n-1}, a_n]]]+s\} \cdots\},\\
    \vdots\\
    =& \underbrace{[a_1, [a_2, [a_3,\cdots, [a_{n-1}, a_{n}]]]\cdots]}_{\text{=0, since the underlying Lie algebra is nilpotent of index n} }+s\\
    =& s.
\end{align*}
Next, we show that $s\in E.$ Let $x\in \mathfrak{a}$, $$\{x,s\}=\underbrace{[x,s]}_{=o,\text{ as ad}_s=o} +s=s,$$ 
\text{and}
$$\{s, x\}=\underbrace{[s,x]}_{=o,\text{ as ad}_s=o}+s=s.$$

This completes the proof.
\end{proof}
\begin{remark}
\begin{itemize}
\item[(i)] From the set $E$ in (\ref{set E}) and the definition of center $\mathfrak{z(\mathfrak{a})}$ of $\mathfrak{a}$, we have $E\subseteq \mathfrak{z(\mathfrak{a})}$.

\item[(ii)] Let $(\mathfrak{a},\{-,-\})$ be a non-empty affinely nilpotent affgebra, then $\mathfrak{z(\mathfrak{a})}$ is non-empty.

\item[(iii)]
Let $(\mathfrak{a},\{-,-\})$ be an affinely nilpotent affgebra and $\mathfrak{I}$ be a non-empty ideaf of $\mathfrak{a}$, then $\mathfrak{I}\cap \mathfrak{z(\mathfrak{a})}$ is non-empty.
\end{itemize}
\end{remark}

\begin{definition}(Nilpotency of an Affine map)

Let $\mathfrak{a}$ be an affine space. An affine map $f: \mathfrak{a}\rightarrow \mathfrak{a}$ is said to be an affinely nilpotent of index $n$ if $f^{n}=f^{n+k}=\text{constant}, \forall~ k\in \mathbb{N}$.
\end{definition}
\begin{theorem}(Nilpotency of Affine map $\iff$ Nilpotency of the corresponding Linear map)\label{NilpotentAffine1}

Let $\mathfrak{a}$ be an affine space. Let $f:\mathfrak{a}\rightarrow \mathfrak{a}$ be an affine map of affinely nilpotent index $n$ if and only if the corresponding linear map $\hat{f}:T_{o}\mathfrak{a}\rightarrow T_{o}\mathfrak{a}$ is a nilpotent map of index $n,$ where $o\in \mathfrak{a}$.
\end{theorem}
\begin{proof}
    Let $f$ is an affine map of affinely nilpotent index $n$, then $f(x)=\hat{f}(x)+f(o)$, where $o\in \mathfrak{a}$, $\forall~ x\in \mathfrak{a}$. 
    \begin{align*}
        f^2(x)&=\hat{f}^2(x)+\underbrace{\hat{f}(f(o))+f(o)}_{\text{constant}},\\
        f^3(x)&=\hat{f}^3(x)+\underbrace{\hat{f}^2(f(o))+\hat{f}(f(o))+f(o)}_{\text{constant}},\\
        &\vdots\\
        f^n(x)&=\hat{f}^n(x)+\underbrace{\hat{f}^{n-1}(f(o))+\cdots +\hat{f}(f(o))+f(o)}_{\text{constant}}.
    \end{align*}
    Since $f^n(x)=\text{constant}$, then $\hat{f}^n(x)=\text{constant}, \forall~ x\in \mathfrak{a}$. Choose $x=o,$ then we get the desired result. For converse part, from the above computation if $\hat{f}^n(x)=o$, then $f^n(x)=\text{constant}$.
\end{proof}

\subsection{ILLUSTRATIVE EXAMPLES}
\subsubsection{ABELIAN LIE  AFFGEBRA}
Let $\mathfrak{a}(\mathfrak{g};0,0,s)$ be an abelian Lie affgebra of special type of affine Lie bracket of the form $\{a,b\}=s$ and $o$ be any element in $\mathfrak{a}$. 
We show that, $\mathfrak{a}$ is a nilpotent Lie affgebra.\\
$\therefore \{a, b\}=s\in \{\mathfrak{a}, \mathfrak{a}\}(=\mathfrak{a}^{(1)})$ as $\mathfrak{a}$ is an abelian Lie affgebra.\\
and $\{c,\{a,b\}\}=\{c,s\}=s\in \{\mathfrak{a}, \mathfrak{a^{(1)}}\}(=\mathfrak{a}^{(2)})$ as $\mathfrak{a}$ is an abelian Lie affgebra. Since $a, b, c \in \mathfrak{a}$ are all arbitrary, thus, the typical elements of $\mathfrak{a}^{(1)}, \mathfrak{a}^{(2)},\mathfrak{a}^{(3)}, \mathfrak{a}^{(4)},\ldots$ are all equal (as in a similar fashion, we can compute $\mathfrak{a}^{(i)}$). Hence, there exist $n=1\in \mathbb{N}$ such that $\mathfrak{a}^{(1)}=\mathfrak{a}^{(1+k)}=E$, for all $k\in \mathbb{N}.$ Thus, $\mathfrak{a}$ is affinely nilpotent Lie affgebra.

\subsubsection{AFFINE NILPOTENCY OF HEISENBERG LIE AFFGEBRA}\label{oneside}
Let $(\mathbb{H}, [-, -])$ be Heisenberg nilpotent Lie algebra of dimension $3$ and $\kappa, \lambda \in$ End$(\mathbb{H})$ such that, for all $x, y \in \mathbb{H}$, 
\begin{equation}\label{gender}
    \lambda([x, y])=[\lambda(x), y]+[x, \lambda(y)]-[x, \kappa(y)]
\end{equation}
Then, for all $s\in \mathbb{H}$, $\mathbb{H}$ is a Lie affgebra with the affine space structure 
\begin{equation}
    <x, y, z>=x-y+z, \alpha\triangleright_{x}y=(1-\alpha)x+\alpha y
\end{equation} 
and the affine Lie bracket, for all $x, y \in \mathbb{H},$
\begin{equation}
    \{x, y\}=[x, y]+\kappa(x)+\lambda(y-x)+s.
\end{equation}
 We denote the Heisenberg Lie affgebra by $\mathfrak{a}(\mathbb{H}; \kappa, \lambda, s).$
 \begin{enumerate}
     \item \label{1.} Let $\kappa=\lambda=0, \text{and}~ \text{ad}_s=0$, then the affine Lie bracket becomes $\{x, y\}=[x, y]+s \in \{\mathfrak{a}, \mathfrak{a}\}(=\mathfrak{a}^{(1)})$. Observe that

$(i)$ $\{z, \{x, y\}\}=\{z, [x, y]+s\}=[z, [x, y]+s]+s=[z,[x, y]]+[z, s]+s=0+[z, s]+s=[z, s]+s=ad_{s}(-z)+s=s$ (as $[z,[x, y]]=0$ for all $x, y, z \in \mathbb{H}$) $ \in  \{\mathfrak{a}, \mathfrak{a}^{(1)}\}(=\mathfrak{a}^{(2)}).$

$(ii)$ $\{w, \{z, \{x, y\}\}\}=\{w, [z, s]+s\}$(by $(i)$)$=[w, [z, s]]+s=[w, [z, s]]+[w, s]+s =0+[w, s]+s=[w, s]+s =ad_{s}(-w)+s=s$ (as $[z,[x, y]]=0$ for all $x, y, z \in \mathbb{H}$) $ \in \{\mathfrak{a}, \mathfrak{a}^{(2)}\}(=\mathfrak{a}^{3})$.

$(iii)$ $\{v, \{w, \{z, \{x, y\}\}\}\}=\{v, [w, s]+s\}$(by $(ii)$)$=[v, [w, s]]+s=[v, [w, s]]+[v, s]+s =0+[v, s]+s=[v, s]+s=ad_{s}(-v)+s=s$ (as $[z,[x, y]]=0$ for all $x, y, z \in \mathbb{H}$)  $\in \{\mathfrak{a}, \mathfrak{a}^{(3)}\}(=\mathfrak{a}^{(4)})$.

 \hspace{7cm}\vdots 

Since $x, y, z, w, v\in \mathbb{H}$ are all arbitrary and $E=\{s\}$, thus, the typical elements of $\mathfrak{a}^{(2)}, \mathfrak{a}^{(3)}, \mathfrak{a}^{(4)}$, $\ldots$ are all equal to the set $E$. Hence, there exist $n=2 \in \mathbb{N}$ such that $\mathfrak{a}^{(2+k)}=\mathfrak{a}^{(2)}=E,$ for all $k \in \mathbb{N}$. Thus, $\mathfrak{a}(\mathbb{H}; 0, 0, s) ~\text{with ad}_s=0$ is an example of affinely nilpotent Lie affgebra.
 \end{enumerate}

\subsection{AFFINE ADJOINT MAP}

\begin{definition}

 A \textbf{representation} of a Lie affgebra $\mathfrak{a}$ is a affine map $\phi$: $\mathfrak{a}\rightarrow Aff(\mathfrak{a})$, where $Aff(\mathfrak{a})$ is a Lie affgebra of all affine endomorphisms with the pointwise heap operation i.e., $<f, g, h>(a):=<f(a), g(a), h(a)>$, the action $(\alpha \triangleright_{f}g)(a):=\alpha\triangleright_{f(a)}g(a)$, and affine Lie bracket $\{f, g\}:=<fg, gf, g>,$ for all $\alpha \in \mathbb{K}, f, g \in Aff(\mathfrak{a}), a\in \mathfrak{a}$.
\end{definition}

Let us define a map $ad^{aff}:\mathfrak{a}\rightarrow Aff(\mathfrak{a})$ which sends $a \mapsto ad^{aff}_{a},$ where $ad^{aff}_{a}(b):=\{a, b\}$. Since $\{a,-\}$ is an affine map, thus, $ad^{aff}_{a}=\{a, -\}$ is also an affine map.
We show that the map $ad^{aff}_{a}$ can be written as sum of a linear map and a constant.\\

Therefore, $ad^{aff}_{a}(b)=\{a, b\}=[a, b]+\kappa(a)+\lambda(b-a)+s=[a, b]+\kappa(a)+\lambda(b)-\lambda(a)+s=ad_{a}(b)+\lambda(b)+(\kappa-\lambda)(a)+s=\underbrace{(ad_{a}+\lambda)}_{linear} (b)+\underbrace{(\kappa-\lambda)(a)+s}_{constant}.$

\begin{proposition}
    Let us consider the above map $ad^{aff}:\mathfrak{a}\rightarrow Aff(\mathfrak{a})$ which sends $a \mapsto ad^{aff}_{a},$ where $ad^{aff}_{a}(b):=\{a, b\}$. Then $ad^{aff}:\mathfrak{a}\rightarrow Aff(\mathfrak{a})$ is an affine map.
\end{proposition}

\begin{proof}
For all $a, b, c \in \mathfrak{a}$ and $\alpha \in \mathbb{K}$, we only need to show that 
\begin{align}
  &ad^{aff}(<a, b, c>)=<ad^{aff}(a), ad^{aff}(b), ad^{aff}(c)>, \label{eq26} \\ 
 &ad^{aff}(\alpha \triangleright_{a}b)=\alpha \triangleleft_{ad^{aff}(a)}ad^{aff}(b).\label{eq27}
\end{align}

Left hand side of $\ref{eq26}=ad^{aff}_{<a, b, c>}$. Let $d \in \mathfrak{a}$ be arbitrary, then, 
$ad^{aff}_{<a, b, c>}(d)=~\{<a, b, c>, d\}=~<\{a, d\}, \{b, d\}, \{c, d\}>~ =~<ad^{aff}_{a}(d),ad^{aff}_{b}(d),ad^{aff}_{c}(d)>\\=~<ad^{aff}_{a}, ad^{aff}_{b}, ad^{aff}_{c}>(d)=<ad^{aff}(a), ad^{aff}(b), ad^{aff}(c)>(d)$. Since $d$ was arbitrary, hence, $ad^{aff}(<a, b, c>)~=~<ad^{aff}(a), ad^{aff}(b), ad^{aff}(c)>~=$ Right hand side of $\ref{eq26}$.

Left hand side of $\ref{eq27}=ad^{aff}_{\alpha \triangleright_{a}b}$. Let $e\in \mathfrak{a}$ be arbitrary. Then, $ad^{aff}_{\alpha \triangleright_{a}b}(e)=\{\alpha \triangleright_{a}b, e\}=\alpha \triangleright_{\{a, e\}}\{b, e\}= \alpha \triangleright_{ad^{aff}_{a}(e)}ad^{aff}_{b}(e)=(\alpha \triangleright_{ad^{aff}_{a}}ad^{aff}_{b})(e)=(\alpha \triangleright_{ad^{aff}(a)}ad^{aff}(b))(e)$. Since $e$ was arbitrary, hence, $ad^{aff}(\alpha \triangleright_{a}b)=\alpha_{ad^{aff}(a)}ad^{aff}(b)$. This completes the proof. \end{proof}
This $\text{ad}^{aff}$ is called the left adjoint affine map. Similarly, we can define the right adjoint affine map.\\
In the case of Lie algebra, derivation is equivalent to Leibniz identity. If we want to find the relationship between $ad^{aff}_{a}(\{x, y\})=\{a, \{x, y\}\}, \{ad^{aff}_{a}(x), y\}=\{\{a, x\}, y\}$ and $\{x, ad^{aff}_{a}(y)\}=\{x, \{a, y\}\}$, interestingly they satisfy Leibniz identity. For more details, see in \cite{Brz26}.


\section{ON ENGLE'S THEOREM FOR LIE AFFGEBRAS}
In this final section, we explore Engle's theorem for Lie affgebras of the form $\mathfrak{a}(\mathfrak{g}; 2\lambda,\lambda,s)$. 
\begin{definition}
    An element $a\in \mathfrak{a}(\mathfrak{g};\kappa, \lambda,s)$ is said to be affinely ad-nilpotent if there exists an natural number $n$ such that $(ad^{aff}_{a})^{n+k}(x)=(ad^{aff}_{a})^{n}(x)=E(\text{constant})$, for all $x\in \mathfrak{a}$ and for all $ k\in \mathbb{N}$. 
\end{definition}
\begin{proposition}
Let $\mathfrak{a}(\mathfrak{g}; \kappa, \lambda, s)$ is Lie affgebra which is affinely nilpotent, then all elements of $\mathfrak{a}(\mathfrak{g}; \kappa, \lambda, s)$ are affinely ad-nilpotent.
\end{proposition}

\begin{proof}
   
Since $\mathfrak{a}(\mathfrak{g}; \kappa, \lambda, s)$ is affinely nilpotent Lie affgebra, then for the sequence of (product) ideafs $(\mathfrak{a}^{(m)})_{m \in \mathbb{Z}_{\geq 0}}$ of $\mathfrak{a}(\mathfrak{g}; \kappa, \lambda, s)$  defined by $\mathfrak{a}^{(0)}:=\mathfrak{a}$, $\mathfrak{a}^{(1)}:=\{\mathfrak{a}, \mathfrak{a}\}$, $\mathfrak{a}^{(2)}:=\{\mathfrak{a},\mathfrak{a}^{(1)}\}$, $\ldots,$ $\mathfrak{a}^{(m)}:=\{\mathfrak{a}, \mathfrak{a}^{(m-1)}\},\cdots$ there exists an integer $n$ such that $\mathfrak{a}^{(n+k)}=\mathfrak{a}^{(n)}=E$ for all $k\in \mathbb{N}$. 
Then, $(ad^{aff}_{a})^{2}(b)=(ad^{aff}_{a}\circ ad^{aff}_{a})(b)=\{a, \{a, b\}\}$. Similarly, $(ad^{aff}_{a})^{n}(b)=(ad^{aff}_{a}\circ \cdots \circ ad^{aff}_{a})(b)=\{a, \{a, \ldots \{a, b\}\}$, where $\{a, b\}=[a, b]+\kappa(a)+\lambda(b-a)+s$ and $\lambda([a, b])=[\lambda(a), b]+[a, \lambda(b)]-[a, \kappa(b)]$, for all $a, b \in \mathfrak{g}.$ For the some $t(\geq n)$ which also works for the Lie affgebra $\mathfrak{a}(\mathfrak{g}; \kappa, \lambda, s)$, $(ad^{aff}_{a})^{t}(b)=(ad^{aff}_{a})^{t+k}(b)=E$, for all $k\in \mathbb{N}$, for all $b \in \mathfrak{a}$. 
\end{proof}

\begin{example}
     Following \eqref{1.} of \eqref{oneside}, $\mathfrak{a}(\mathbb{H}; 0, 0, s)~ \text{with ad}_s=0$ is affinely nilpotent Lie affgebra of index $2$. Let us consider $ad^{aff}_{a}(b)=\{a, b\}$. We show that all elements of $\mathfrak{a}(\mathbb{H}; 0, 0, s)$ are affinely ad-nilpotent. Let $a\in \mathfrak{a}$. We only need to show that $ad^{aff}_{a}$ is affinely nilpotent. Thus, $(ad^{aff}_{a})^{2}(b)=\{a, \{a, b\}\}=\{a, [a, b]+s\}=[a, [a, b]]+[a, s]+s=0+[a, s]+s=s-ad_{s}(a)=s$, and $(ad^{aff}_{a})^{3}(b)=ad^{aff}_{a}([a,s]+s)=\{a, [a, s]+s\}=[a, [a, s]]+[a, s]+s=0+[a, s]+s=[a, s]+s=s-ad_{s}(a)=s$. In a similar fashion, it can be obtained that $(ad^{aff}_{a})^{2}(b)=(ad^{aff}_{a})^{2+k}(b)=E=\{s\}=constant$, for all $k\in \mathbb{N}$.
\end{example}
What about the converse? If all elements of $\mathfrak{a}(\mathfrak{g}; \kappa, \lambda, s)$ are affinely ad-nilpotent, does it imply that $\mathfrak{a}(\mathfrak{g}; \kappa, \lambda, s)$ is affinely nilpotent Lie affgebra? Note that the result is true for $\kappa=\lambda=0$ and $s=0$, known as Engel's theorem for Lie algebras. 
\begin{definition}
Let $\mathfrak{a}$ be any Lie affgebra and $o$ be any element in $\mathfrak{a}$. Then there exists $\kappa, \lambda, s$ such that $\mathfrak{a}=\mathfrak{a}(T_{o}\mathfrak{a};\kappa, \lambda, s)$. The dimension of the Lie affgebra $\mathfrak{a}$ means to be the dimension of $T_{o}\mathfrak{a}$.
\end{definition}
\begin{remark}
    The above definition of dimension of Lie affgebra makes sense since different choices of $o$ lead to isomorphic tangent spaces.
\end{remark}

Recall that elements of the quotient Lie affgebra $\mathfrak{a}/\mathfrak{Z(a)}$ is of the form $$\overline{a}:=\{<a, b', b>|\forall~ b\in \mathfrak{Z(a)}, \exists b'\in \mathfrak{Z(a)}\}, \text{where $a\in \mathfrak{a}$}.$$
Let us consider the following set $$\overline{E}:=\{\overline{x}\in \mathfrak{a}/\mathfrak{Z(a)}~|~ \{\overline{x}, \overline{a}\}=\{\overline{a}, \overline{x}\}=\overline{x}, \forall~ \overline{a}\in \mathfrak{a}/\mathfrak{Z(a)}\}.$$
\begin{proposition}\label{1}
    For any $x\in \mathfrak{a}$, $\overline{x}\in \overline{E}$ if and only if $x\in E.$
\end{proposition}
\begin{proof}
    Let $\overline{x}\in \overline{E}$, then $\{\overline{x}, \overline{y}\}=\{\overline{y}, \overline{x}\}=\overline{x},  \forall~ \overline{y} \in \mathfrak{a}/\mathfrak{Z(a)}$. Observe that
    \begin{align*}
        &\{\overline{x}, \overline{y}\}=\overline{x},\\
        & \{<x, s_1, s_2>, <y, s_3, s_4>\}\in \overline{x}, \forall~ s_2, s_4 \in \mathfrak{Z(a)}, \exists s_1, s_3 \in \mathfrak{Z(a)},\\
        & <\{x, y\}, s_5, s_6> \in \overline{x}, \text{for some $s_5, s_6 \in \mathfrak{Z(a)}.$}
    \end{align*}
    Then by definition of $\overline{x}$, we get $\{x, y\}=x$, for all $y\in \mathfrak{a}$. Similarly, we can obtain $\{y, x\}=x$, for all $y\in\mathfrak{a}.$ Hence, $x\in E$.\\
    Let $x\in E.$ Then $\{x, y\}=\{y, x\}=x, \forall~ y \in \mathfrak{a}$. Since $\overline{\{x, y\}}=\{\overline{x}, \overline{y}\}$, thus $\overline{\{x, y\}}=\{\overline{x}, \overline{y}\}=\overline{\{y, x\}}=\{\overline{y}, \overline{x}\}=\overline{x}.$ Hence, $\overline{x}\in \overline{E}.$
\end{proof}
\begin{proposition}\label{2}
    $E\subset \mathfrak{a}^{(n)}, \forall~ n\geq 1.$
\end{proposition}
\begin{proof}
    Let $x\in E$, then $\{x, y\}=\{y, x\}=x, \forall~ y\in \mathfrak{a}$. In particular, we may choose $y=x$, thus $\{x,x\}=x.$ Then $$x=\{x,\{x,\{x,\{\cdots x\underbrace{ \}\}\cdots \}}_{n\text{-many}}\in \mathfrak{a}^{(n)}$$. Thus, we are done.
\end{proof}

Henceforth, we assume that $\kappa=2\lambda$ and $[x, \lambda(y)]=0,~ \text{for all}~ x, y\in \mathfrak{a},$ and $\text{ad}_s=0$ for a Lie affgebra $\mathfrak{a}(\mathfrak{g}; \kappa, \lambda, s)$.
\begin{lemma}\label{lemma 9.2}
    Let $\mathfrak{a}$ be a Lie affgebra and $\mathfrak{Z(a)}$ be the center of the Lie affgebra $\mathfrak{a}$. If $\mathfrak{a/Z(a)}$ is an affinely nilpotent Lie affgebra, then $\mathfrak{a}$ is also an affinely nilpotent Lie affgebra.
\end{lemma}
\begin{proof}
    Since $\mathfrak{a/Z(a)}$ is affinely nilpotent Lie affgebra, then there exists $n$ such that $(\mathfrak{a/Z(a)})^{(n)}=(\mathfrak{a/Z(a)})^{(n+k)}=E$, $\forall k\in \mathbb{N}$. We prove the elements of $(\mathfrak{a/Z(a)})^{n}$ can be thought of as elements of $\mathfrak{a^{(n)}/Z(a)}.$ We start with any element $\overline{a_{i}}=\{<a_{i}, z_{i}, z'_{i}>| \forall z'_{i}\in \mathfrak{Z(a)}, \exists z_i \in \mathfrak{Z(a)}\}$.\\
    Consider $\{\overline{a_{1}},\{\overline{a_{2}},\overline{a_{3}}\}\}$ and start by doing for representative 
    \begin{align*}
    &\{[a_{1}, z_1, z'_1], \{[a_2, z_2, z'_2], [a_3, z_3,z'_3]\}\}=\{[a_{1}, z_1, z'_1],[\{a_2,[a_3, z_3, z'_3]\},\underbrace{ \{z_2, [a_3, z_3, z'_3]\}}_{z_4 \in \mathfrak{Z(a)}}, \underbrace{\{z'_2, [a_3, z_3, z_3']\}}_{z_5 \in \mathfrak{Z(a)}}\}\tag*{}\\
    =&\{[a_1, z_1, z'_1], [\{a_2, [a_3, z_3, z'_3]\}, z_4, z_5]\}\tag*{}\\
    =&\{[a_1, z_1, z'_1],[\{a_2, a_3\}, \underbrace{ \{a_2, z_3\}}_{z_6 \in \mathfrak{Z(a)}}, \underbrace{\{a_2, z'_3\}}_{z_7 \in \mathfrak{Z(a)}}], z_4, z_5\}\tag*{}\\
    =&\{[a_1, z_1, z'_1],[\{a_2, a_3\}, z_6, z_7], z_4, z_5\}\tag*{}\\
    =&\{[a_1, z_1, z'_1],[\{a_2, a_3\}, z_6, z_7], z_4, z_5\}\tag*{}\\
    =&\{[a_1, z_1, z'_1],[\{a_2, a_3\}, z_6, z_8]\}, \text{where $z_8=[z_7, z_4, z_5]$}\tag*{}\\
    =&[[\{a_1,\{a_2, a_3\}\}, \underbrace{\{a_1, z_6\}}_{z_9}, \underbrace{\{a_1, z_8\}}_{z_{10}}], \underbrace{[\{z_1, \{a_2, a_3\}\},\{z_1, z_6\}, \{z_1, z_8\}]}_{z_{11}}, \underbrace{[\{z'_1, \{a_2, a_3\}\}, \{z'_1, z_6\}, \{z'_1, z_8\}]}_{z_{12}}]\tag*{}\\
    =&[\{a_1,\{a_2, a_3\}\}, z, z'] \text{ for some $ z,z'\in \mathfrak{z(a)}$(on simplification).}\tag*{}
\end{align*} 
Similarly, we have $\{\overline{a_{1}},\{\overline{a_{2}},\{\cdots,\overline{a_{n}}\underbrace{\}\}\}\cdots\}}_{n~\text{terms}}=[\{a_{1},\{a_{2}\{\cdots a_{n}\underbrace{\}\}\}\cdots\}}_{n~\text{terms}},z',z'']$ for some $z',z''\in \mathfrak{z(a)}$ is an element of $\mathfrak{a^{n}/Z(a)}$. That is;
\begin{equation*}
(\mathfrak{a}/\mathfrak{Z(a)})^{(n)}=\mathfrak{a}^{(n)}/ \mathfrak{Z(a)} \quad \text{using the fact that the center $\mathfrak{z(a)}$, is an ideaf}.   
\end{equation*}
     
     Using the Proposition (\ref{1}) and $\mathfrak{a}^{(n)}\subset \mathfrak{a}$, we obtain, for any $x\in\mathfrak{a}^{(n)}\subset \mathfrak{a}$ $\text{and}$ $\overline{x}\in \overline{E},$ $\text{thus}$ $ x\in E.$ Hence, $\mathfrak{a}^{(n)}\subset E$ and by the Proposition (\ref{2}), we get $\mathfrak{a}^{(n)}=E.$ Consider $\mathfrak{a}^{(n+1)}=\{\mathfrak{a}, \mathfrak{a}^{(n)}\}=\{\mathfrak{a}, E\}=E$, i.e., $\mathfrak{a}^{(n+1)}=\mathfrak{a}^{(n)}=E.$ Then for all $k\geq 1$, we get $\mathfrak{a}^{(n+k)}=\mathfrak{a}^{(n)}=E.$ Thus, $\mathfrak{a}$ is affinely nilpotent Lie affgebra.
\end{proof}
\begin{lemma}\label{lemma 9.3}
    Let $\mathfrak{b}$ be a subset of $\text{aff}(\mathfrak{a})$ and $o\in \mathfrak{a}$ be any element. Suppose that $T_{o}\mathfrak{a}(\neq \lbrace o \rbrace)$ is a finite dimensional vector space with origin $o$. If $\mathfrak{b}$ consists of affinely nilpotent endomorphisms of $\mathfrak{a}$, then there exists an element $v(\neq o)\in T_{o}\mathfrak{a}$ such that $\mathfrak{b}.v=\{f(o)~|~f\in \text{aff}(\mathfrak{a})\}.$
\end{lemma}
\begin{proof}
    Let $f\in \mathfrak{b}$ implies $f\in aff(\mathfrak{a})$, then there exists a linear map $\hat{f}:T_{o}\mathfrak{a}\rightarrow T_{a}\mathfrak{a}$ such that $\hat{f}(a)=f(a)-f(o)$ if and only if $f(a)=\hat{f}(a)+f(o)$. Thus, by the Theorem \eqref{NilpotentAffine1}, $\{\hat{f}\}_{f\in \mathfrak{b}}$ is a collection of nilpotent linear operators on $T_{o}\mathfrak{a}$, say it $L$, i.e. $L:=\{\hat{f}:T_o\mathfrak{a}\rightarrow T_{o}\mathfrak{a} ~|~ \hat{f}$ is a nilpotent linear map associated to $f\in \mathfrak{b}\}.$ Now, by Theorem $3.3$ in $\cite{Hum}$, then there exists $v(\neq o)\in T_{o}\mathfrak{a}$ such that $L.v=o$ i.e. $\hat{f}(v)=o$, $\forall \hat{f}\in L$. Thus, we obtained  $f(v)=f(o)$. Hence, $\mathfrak{b}.v=\{ f(o)~|~ f\in \mathfrak{b}\}.$
\end{proof}
\begin{theorem}
    If all elements of $\mathfrak{a}(\mathfrak{g}; \kappa=2\lambda, \lambda, s)$ are affinely ad-nilpotent, then $\mathfrak{a}(\mathfrak{g}; \kappa=2\lambda, \lambda, s)$ is affinely nilpotent Lie affgebra.
\end{theorem}
\begin{proof}
    Let $\Bar{\mathfrak{a}}=ad^{aff}(\mathfrak{a})$ denote the image of $\mathfrak{a}$ under the affine adjoint map  $ad^{aff}:\mathfrak{a}\rightarrow Aff(\mathfrak{a})$. By hypothesis, every element of $\Bar{\mathfrak{a}}=ad^{aff}(\mathfrak{a})$ is affinely nilpotent map of $\mathfrak{a}$. Thus $\bar{\mathfrak{a}} \subseteq Aff(\mathfrak{a})$ consists of affinely nilpotent endomorphism of Lie affgebra $\mathfrak{a}$. Assume $\kappa=2\lambda$ and we start the proof for left affine adjoint. Then by \ref{lemma 9.3} there exist $v\in T_{o}\mathfrak{a}$ such that $\bar{\mathfrak{a}}. v=ad^{aff}(\mathfrak{a}).v=\{ad^{aff}_{a}(o) ~|~ ad^{aff}_{a} \in ad^{aff}(\mathfrak{a}), a\in \mathfrak{a}\}$ and $ad^{aff}_{a}(v)=ad^{aff}_{a}(o)$ implies $\{a, v\}=\{a, o\}=[a, o]+\kappa(a)-\lambda(a)+s=2\lambda(a)-\lambda(a)+s=\lambda(a)+s,$ and for right adjoint, by Lemma \ref{lemma 9.3}, there exist $v\in T_{o}\mathfrak{a}$ such that $\bar{\mathfrak{a}}. v=ad^{aff}(\mathfrak{a}).v=\{ad^{aff}_{a}(o) ~|~ ad^{aff}_{a} \in ad^{aff}(\mathfrak{a}), a\in \mathfrak{a}\}$ and $ad^{aff}_{a}(v)=ad^{aff}_{a}(o)$ implies $\{v, a\}=\{o, a\}=[o, a]+\kappa(o)+\lambda(a-o)+s=\lambda(a)+s.$ Thus, we obtain $\{a, v\}=\{v, a\}$ for all $a\in \mathfrak{a}$. This shows that $(o\neq)v\in \mathfrak{Z(a)}$. Thus, the dimension of $T_{o}\mathfrak{Z(a)}$ is greater than equal to $1$. Since $T_{o}(\mathfrak{a/Z(a)})\simeq T_{o}\mathfrak{a}/T_{o}\mathfrak{Z(a)}$, where $o\in \mathfrak{Z(a)}$, thus, the dimension of $\mathfrak{a/Z(a)}$ is less than dimension of $\mathfrak{a}$ and by using induction on dimension of $\mathfrak{a}$, we find that $\mathfrak{a/Z(a)}$ is affinely nilpotent Lie affgebra. Then by Lemma \ref{lemma 9.2}, we conclude that $\mathfrak{a}$ is an affinely nilpotent Lie affgebra.    
\end{proof}

\begin{center}
 {\bf ACKNOWLEDGEMENT}
 \end{center}
 
This research was supported by the Core Research Grant (CRG) of the Anusandhan National Research Foundation (ANRF), formerly the Science and Engineering Research Board (SERB), under the Department of Science and Technology (DST), Government of India (Grant No.~CRG/2022/005332). The authors gratefully acknowledge the financial support received from the aforementioned agency. Tarik Anowar also gratefully acknowledges financial support from the UGC-JRF (MANF), while Sayan Thokdar gratefully acknowledges financial support from the UGC-JRF.

\end{document}